\let\accentvec\vec
\RequirePackage{fix-cm}
\documentclass[smallextended]{svjour3}       
\journalname{Journal}
\smartqed  
\usepackage{graphicx}
\usepackage{longtable}
\usepackage{pdflscape}
\usepackage{epsfig}
\usepackage{lscape}
\usepackage{rotating}

\let\vec\accentvec
\usepackage{latexsym}
\usepackage{amssymb}
\usepackage{amsmath}
\usepackage{amsfonts}
\usepackage{paralist}
\usepackage{enumerate}
\usepackage[normalem]{ulem}
\usepackage{color}
\usepackage[colorlinks,linkcolor=blue,anchorcolor=blue,citecolor=blue]{hyperref}
\usepackage{url}

\usepackage[left=32mm,right=32mm,top=35mm,bottom=35mm,a4paper]{geometry}

\numberwithin{theorem}{section}
\numberwithin{lemma}{section}
\numberwithin{proposition}{section}
\numberwithin{equation}{section}
\numberwithin{remark}{section}

\DeclareMathOperator*{\argmin}{arg\,min}
\DeclareMathOperator*{\dist}{dist}
\DeclareMathOperator*{\ri}{ri}

\DeclareMathOperator*{\diag}{Diag}
\DeclareMathOperator*{\dom}{dom}
\def\prox{{\rm Prox}}
\def\A{{\mathcal A}}
\def\B{{\mathcal B}}
\def\D{{\mathcal D}}
\def\E{{\mathcal E}}
\def\F{{\mathcal F}}
\def\G{{\mathcal G}}
\def\H{{\mathcal H}}
\def\I{{\mathcal I}}

\def\L{{\mathcal L}}
\def\M{{\mathcal M}}
\def\N{{\mathcal N}}
\def\O{{\mathcal O}}
\def\P{{\mathcal P}}
\def\Q{{\mathcal Q}}
\def\R{{\mathcal R}}
\def\S{{\mathcal S}}
\def\T{{\mathcal T}}
\def\U{{\mathcal U}}
\def\V{{\mathcal V}}
\def\W{{\mathcal W}}
\def\X{{\mathcal X}}
\def\Y{{\mathcal Y}}
\def\Z{{\mathcal Z}}

\def\hat{\widehat}
\def\tilde{\widetilde}

\def\[{\begin{equation}}
\def\]{\end{equation}}

\def\norm#1{\|#1\|}
\def\inprod#1#2{\big\langle#1,\,#2\big\rangle}

\def\disp{\displaystyle}
\def\mc{\multicolumn}
\def\norm#1{\|#1 \|}
\def\inprod#1#2{\langle#1,\,#2 \rangle}
\def\cA{{\cal A}} \def\cB{{\cal B}}  \def\cK{{\cal K}}
 \def\cD{{\cal D}}  
\def\cH{{\cal H}}     
\def\cQ{{\cal Q}}  \def\cN{{\cal N}} \def\cW{{\cal W}}
\def\cS{{\cal S}}     
 \def\cX{{\cal X}} \def\cY{{\cal Y}} \def\cZ{{\cal Z}}
 \def\cN{{\cal N}} \def\cL{{\cal L}} \def\Sn{{\cal S}^n}

\def\diag{{\rm Diag}}
\def\Range{\textup{Range}}
\def\dist{\textup{dist}}

\begin{document}
\title{\bf \Large An Efficient Inexact Symmetric Gauss-Seidel Based Majorized ADMM for High-Dimensional Convex Composite Conic Programming
\thanks{The research of the first author was supported by the China Scholarship Council while visiting the National University of Singapore and the National Natural Science Foundation of China (Grant No. 11271117).
The research of the second and the third authors was supported in part by the Ministry of Education, Singapore, Academic Research Fund (Grant No. R-146-000-194-112).}}
\titlerunning{Inexact sGS based Majorized ADMM}
\author{Liang Chen \and Defeng Sun \and Kim-Chuan Toh}
\institute{Liang Chen
\at College of Mathematics and Econometrics, Hunan University, Changsha, 410082, China.\\
\email{chl@hnu.edu.cn}
\and
Defeng Sun
\at Department of Mathematics and Risk Management Institute, National University of Singapore, 10 Lower Kent Ridge Road, Singapore.\\
\email{matsundf@nus.edu.sg}
\and
Kim-Chuan Toh, Corresponding author
\at Department of Mathematics, National University of Singapore, 10 Lower Kent Ridge Road, Singapore.\\
\email{mattohkc@nus.edu.sg}
}
\date{\large May 30 2015, revised on March 17,2016}
\maketitle

\begin{abstract}
In this paper, we propose an {\em inexact} multi-block ADMM-type first-order
method  for solving a class of high-dimensional convex composite conic optimization problems to moderate accuracy. The design of this method combines an inexact 2-block majorized semi-proximal ADMM and the recent advances in the  {inexact} symmetric Gauss-Seidel (sGS) technique
 for solving a multi-block convex composite quadratic programming whose objective
contains a nonsmooth term involving only the first
block-variable.
One distinctive feature of our proposed method (the sGS-imsPADMM) is that it only needs one cycle of an inexact sGS method, instead of an unknown number of
cycles,
to solve each of the subproblems involved.
With some simple and implementable error tolerance criteria, the cost for solving the subproblems can be
 greatly  reduced,
and many steps in the forward sweep of each sGS cycle can often be
skipped,
which further contributes to the efficiency of the proposed method.
Global convergence as well as the iteration complexity in {the non-ergodic sense is}
established.
Preliminary numerical experiments on some high-dimensional linear and convex quadratic SDP problems with a large number of
 linear equality and inequality constraints
are also provided.  The results show that for the vast majority of the tested problems,  the sGS-imsPADMM is 2 to 3 times faster than the directly extended multi-block ADMM with  the aggressive step-length of  $1.618$, which
is currently the benchmark
among first-order methods for solving multi-block linear and quadratic SDP problems though its convergence is not guaranteed.
\end{abstract}

\section{Introduction}
The objective of this paper is to design an efficient first-order method for solving the following high-dimensional convex composite quadratic conic programming problem to moderate accuracy:
\begin{equation}
\label{eq:cqcp}
   \begin{array}{c}
    \min \big\{ \theta(x) + \frac{1}{2}\inprod{x}{\cQ x} + \inprod{c}{x}
    \mid  \cA x - b = 0, \, x\in\cK \big\},
    \end{array}
 \end{equation}
where $\theta: \mathcal{X} \to (-\infty,+\infty]$ is a closed proper convex function,
$\cQ:\X\to\X$ is a self-adjoint positive semidefinite linear operator,
$\cA:\cX\to\cY$ is a linear map,
$c\in\cX$ and $b\in \cY$ are given data, $\cK \subseteq\cX$ is a closed convex cone,
$\mathcal{X}$ and $\mathcal{Y}$ are two real finite dimensional Euclidean spaces each equipped with an inner product
$\langle \cdot, \cdot \rangle$ and its induced norm $\|\cdot\|$.
Here, the phrase ``high-dimension'' means that the linear operators $\A\A^{*}$ and/or $\Q$ are too large to be stored explicitly or to be factorized by the Cholesky decomposition.
By introducing a slack variable $u\in \cX$, one can equivalently recast (\ref{eq:cqcp}) as
\begin{equation}
\label{eq:cqcpp}
\begin{array}{c}
    \min \, \big\{ \theta(u) + \frac{1}{2}\inprod{x}{\cQ x} + \inprod{c}{x}
    \mid
     \cA x - b =0,\, u-x = 0, \, x\in\cK \big\}.
\end{array}
\end{equation}
Then, solving the dual of problem \eqref{eq:cqcpp} is equivalent to solving
\begin{equation}
\label{eq:cqcp_dual}
\begin{array}{c}
\min\big\{\theta^*(-s)+\frac{1}{2}\inprod{w}{\cQ w}-\inprod{b}{\xi}
\mid
s + z - \cQ w + \cA^*\xi = c,
\, z\in\cK^*,\, w\in\cW\big\},
\end{array}
\end{equation}
where $\cW \subseteq \cX$ is any subspace  containing $\Range(\cQ)$,
$\cK^*$ is the dual cone of $\cK$ {defined by
$\cK^*:= \{d\in\cX\;|\; \inprod{d}{x}\ge 0 \; \forall\, x\in \cK\}$},  $\theta^*$  is the Fenchel conjugate of the convex function $\theta$.
Particular examples of \eqref{eq:cqcp} include convex quadratic semidefinite programming (QSDP), convex quadratic programming (QP),  nuclear semi-norm penalized least squares
and robust PCA (principal component analysis) problems. One may refer to \cite{lithesis,li14} and references therein for a brief introduction on these examples.

Let $m$ and $n$ be given nonnegative integers, $\Z$,  $\X_i$, $1\le i\le m$ and $\Y_j$, $1\le j\le n$, be finite dimensional real Euclidean spaces each endowed with an inner product $\langle\cdot,\cdot\rangle$ and its induced norm $\|\cdot\|$. Define $\X:=\X_1\times\ldots\times\X_m$ and $\Y:=\Y_1\times\ldots\times\Y_n$.
Problem \eqref{eq:cqcp_dual} falls within the following general convex composite  programming model:
\[
\label{problem:primal}
\min_{x\in\X,y\in\Y} \big\{
 p_{1}(x_1)+f(x_1,\ldots,x_{m})+q_{1}(y_1)+g(y_{1},\ldots,y_{n})\,
|\, \A^* x+\B^* y=c
\big\},
\]
where $p_{1}:\X_1\to(-\infty,\infty]$ and $q_{1}:\Y_{1}\to(-\infty,\infty]$ are two closed proper convex functions, $f:\X\to(-\infty,\infty)$ and $g:\Y\to(-\infty,\infty)$ are continuously differentiable convex functions whose gradients are Lipschitz continuous.
The linear mappings $\cA:\cX \rightarrow \cY$ and  $\cB:\cX\rightarrow \cZ$ are defined
such that their adjoints are given by  $\cA^*x = \mbox{$\sum_{i=1}^m$} \cA_i^*x_i$
 for  $x=(x_1,\ldots,x_m) \in \X$,
 and   $\cB^*y = \mbox{$\sum_{j=1}^n$} \cB_j^*y_j$ for
$y=(y_1,\ldots,y_n)\in \Y$,
where
$\A_{i}^*:\X_{i}\to\Z,i=1,\dots,m$ and $\B_{j}^*:\Y_{j}\to\Z,j=1,\dots,n$ are the adjoints of the linear maps $\A_{i}:\Z\to\X_{i}$ and $\B_{j}:\Z\to\Y_{i}$ respectively.
For notational convenience, in the subsequent discussions we define the
functions $p:\X\to(-\infty,\infty]$ and $q:\Y\to(-\infty,\infty]$ by $p(x):=p_{1}(x_{1})$ and $q(y):=q_{1}(y_{1})$.
For problem \eqref{eq:cqcp_dual}, one can express it in the generic form
\eqref{problem:primal} by setting
$$
\begin{array}{c}
 p_1(s) = \theta^*(-s), \ \   f(s,w) = \frac{1}{2} \inprod{w}{\cQ w},\ \  q_1(z) = \delta_{\cK^*}(z),\\[3pt]
g(z,\xi) = -\inprod{y}{\xi},\ \ \cA^*(s,w) = s - \cQ w \  \ \mbox{and}\  \ \cB^*(z,\xi) = z + \cA^*\xi.
\end{array}
$$
There are various numerical methods available for solving  problem \eqref{problem:primal}. Among them, perhaps the first choice is
the   augmented Lagrangian method (ALM)  pioneered  by Hestenes \cite{hestenes69}, Powell \cite{powell} and  Rockafellar \cite{rockafellar},
 if one
 does not attempt to exploit  the composite structure  in  \eqref{problem:primal} to gain computational efficiency.
Let $\sigma>0$ be a given penalty parameter.  The augmented Lagrangian function of problem \eqref{problem:primal} is defined as follows: for  any $(x,y,z)\in\X\times\Y\times\Z$,
$$
\L_{\sigma}(x,y;z)
:=p(x)+f(x)+q(y)+g(y)
+\langle z,\A^*x
+\B^*y-c\rangle+\frac{\sigma}{2}\|\A^*x+\B^*y-c\|^2.
$$
Given an initial point  $z^0 \in \cZ$, the ALM consists of the following iterations:
\begin{eqnarray}
 (x^{k+1},  y^{k+1}) &:= & \textup{argmin}
\;\cL_\sigma(x,y;z^k),
\label{eq-auglagl-w}
\\[5pt]
\nonumber
 z^{k+1} &:=&\displaystyle z^k +\tau \sigma  (  \A^* x^{k+1}+\B^* y^{k+1}-c ),
\end{eqnarray}
where $ \tau\in(0,2)$ is the step-length. A key attractive property of the ALM and  its inexact variants, including the inexact proximal point algorithm (PPA) obeying summable error tolerance criteria proposed by Rockafellar \cite{rockafellar,rockafellar76} and the inexact PPAs proposed by  Solodov and Svaiter \cite{solodov99,solodov992,solodov00} using relative error criteria,
 is their   fast local linear convergence property when  the penalty parameter exceeds a certain threshold.
However, it is generally difficult and expensive to solve the inner subproblems in these methods exactly or to high accuracy, especially in high-dimensional settings, due to the coupled quadratic  term interacting with two nonsmooth functions in the augmented Lagrangian function.
By exploiting the composite structure of \eqref{problem:primal}, one may use the block coordinate descent (BCD) method to solve the subproblems in (\ref{eq-auglagl-w}) inexactly.  However, it can also be expensive to adopt such a strategy as the number of BCD-type cycles needed to solve each subproblem to the required accuracy
can be large. In addition, it is also  {computationally not economical} to use the ALM during the early stage of solving problem \eqref{problem:primal} when the fast local linear convergence of ALM has not kicked in.

A natural alternative to the ALM, for solving linearly constrained $2$-block convex optimization problems such as
\eqref{problem:primal}, is the alternating direction method of multipliers (ADMM) \cite{fortin83,Gabay:1976ff,glo80book,glo75}, which solves $x$ and $y$ alternatively in a Gauss-Seidel fashion  {(one may refer to \cite{eck12} for a recent tutorial on the ADMM)}.
Computationally, such a strategy can be beneficial because solving $x$ or $y$ by fixing the other variable in (\ref{eq-auglagl-w}) is potentially easier than solving $x$ and $y$ simultaneously.
Just as in the case for the ALM and PPAs, one may also have to solve the ADMM subproblems approximately. Indeed, for this purpose,
Eckstein and Bertsekas  \cite{eckstein1992} proposed the first  inexact version of the ADMM based on the PPA theory  and    Eckstein  \cite{eckstein1994} introduced a proximal ADMM (PADMM) to make the subproblems easier to solve. The inexact version of Eckstein's PADMM and its extensions   can be found in \cite{he02,li-inexact,ng11}, to name  a few.
These ADMM-type methods are very competitive for solving certain 2-block separable problems and they have been used frequently to generate a good initial point to warm-start the ALM.
However, for many other cases such as the high-dimensional multi-block convex composite conic programming problem (\ref{eq:cqcp}) and its dual (\ref{eq:cqcp_dual}), it can be very expensive to solve the ADMM subproblems (each is a composite problem with smooth and nonsmooth terms in two or more blocks of variables)
to high accuracy. Also, by using BCD-type methods to solve these subproblems, one may still face  the same drawback as in solving the ALM subproblems
by requiring an unknown number of BCD inner iteration cycles.
One  strategy which may be adopted to alleviate the computational burden in solving the subproblems is to
divide the variables in \eqref{problem:primal} into three or more blocks (depending on its composite structure),
and to solve the resulting problems by a multi-block ADMM-type method (by directly extending the 2-block ADMM or PADMM to the multi-block setting).
However, such a directly extended method may be non-convergent as was shown in \cite{chen14}, even if the functions $f$ and $g$ are separable with respect to these blocks of variables, despite ample numerical results showing that it is often practically efficient and effective \cite{sun13}. Thus, different strategies are called for to deal with the numerical difficulty just mentioned.

Our primary objective  in this paper is to construct a multi-block ADMM-type method for solving
high-dimensional multi-block convex composite optimization problems to  medium accuracy
with the essential flexibility that the inner subproblems are allowed to be solved only approximately.
We should emphasize that the flexibility is essential because
this gives us the freedom
to solve large scale linear systems of equations (which typically
arise in high-dimensional problems)
approximately by an iterative solver such as the conjugate gradient method.
Without such a flexibility, one would be forced
to modify the corresponding subproblem by adding an appropriately chosen
``large" semi-proximal term so as to get a closed-form solution for the
modified subproblem. But such a modification can sometimes significantly
slow down the outer iteration as we shall see later in our numerical experiments.

In this paper, we achieve our goal by proposing an inexact
symmetric Gauss-Seidel (sGS) based  majorized
semi-proximal ADMM (we name it as sGS-imsPADMM for easy reference)
for solving \eqref{problem:primal}, for which each of its subproblems only needs one cycle of an inexact sGS iteration instead of an unknown number of cycles.
Our method is motivated by the works of \cite{limajorize} and  \cite{lisgs}
in that
it is developed via a novel integration of  the majorization technique used in \cite{limajorize}
with the inexact symmetric Gauss-Seidel iteration technique proposed in \cite{lisgs} for solving a  convex minimization problem
whose objective is the sum of a multi-block quadratic function and a non-smooth function
involving only the first block. However, non-trivially, we also design checkable inexact minimization criteria for solving the sPADMM subproblems
while still being able to establish the convergence of the inexact method.
Our convergence analysis relies on the key observation that the
results in \cite{sun13} and \cite{lisgs} are obtained via establishing
the equivalence  of their proposed algorithms
to particular cases of the $2$-block sPADMM
in \cite{fazel13} with some specially constructed semi-proximal terms.
Owing to the inexact minimization criteria, the cost for solving the subproblems
in our proposed algorithm can greatly be reduced.
For example, one can now solve a very large linear system of equations via an iterative solver
to an appropriate accuracy instead  of a very high accuracy as required by a method
with no inexactness flexibility.

Moreover,  by using the majorization technique, the two smooth functions $f$ and $g$ in \eqref{problem:primal} are  allowed to be non-quadratic. Thus, the proposed method is capable of  dealing with even more problems beyond the scope of convex quadratic conic programming. The success of our proposed inexact sGS-based ADMM-type method would thus also meet the pressing demand for an efficient algorithm to find a good initial point
to warm-start the augmented Lagrangian method so as to quickly enjoy its fast local  linear convergence.

To summarize, the main contribution of this paper is that by taking advantage of the inexact sGS technique in  \cite{lisgs}, we develop a simple, implementable and efficient inexact   first-order algorithm, the sGS-imsPADMM, for solving high-dimensional multi-block convex composite
conic optimization problems to moderate accuracy. We have also
established the global convergence as well as the non-ergodic iteration complexity
of our proposed method.
Preliminary numerical experiments on the class of high-dimensional linear and convex quadratic SDP problems with a large number of linear
equality and inequality constraints
are also provided.
The results show that on the average,  the sGS-imsPADMM {is 2 to 3 times faster} than the directly extended multi-block ADMM even with
the aggressive step-length of 1.618, which is currently the  benchmark among first-order methods
for solving multi-block linear and quadratic SDPs though its convergence is not guaranteed.

The remaining parts of this paper are organized as follows. In Sec. \ref{preliminary}, we present some preliminary results from convex analysis.
In Sec. \ref{imspadmm}, we propose an inexact two-block majorized sPADMM (imsPADMM), which lays the foundation for later {algorithmic} developments.
In Sec. \ref{sgs-imspadmm}, we give a quick review of the inexact   sGS technique developed
in \cite{lisgs} and propose our  sGS-imsPADMM algorithm
for the multi-block composite optimization problem (\ref{problem:primal}), which constitutes as the main result of this paper. Moreover, we establish the relationship between the sGS-imsPADMM and the imsPADMM to substantially simplify the convergence analysis.
In Sec. \ref{convergence},  we establish the global convergence of the imsPADMM. Hence, the convergence of  the sGS-imsPADMM is also derived.
In Sec. \ref{convergence-rate}, we  study the non-ergodic iteration complexity of the proposed algorithm.
In Sec. \ref{numerical}, we present our numerical results, as well as some efficient computational techniques designed in our implementation. We conclude the  paper in Sec. \ref{conclusion}.

\section{Preliminaries}
\label{preliminary}
Let $\U$ and $\V$ be two arbitrary finite dimensional real Euclidean spaces each endowed with an inner product {$\langle\cdot,\cdot\rangle$  and its induced norm  $\|\cdot\|$}. For any linear map $\O:U\to\V$, we use $\O^*$ to denote its adjoint and $\|\O\|$ to denote its induced norm.
For a self-adjoint positive semidefinite linear operator $\H:\U\to\U$,   there exists a unique self-adjoint positive semidefinite linear operator, denoted as $\H^{\frac{1}{2}}$, such that $\H^{\frac{1}{2}}\H^{\frac{1}{2}}=\H$.
For any $u,v\in\U$,   define $\langle u,v\rangle_{\H}:=\langle u,\H v\rangle$ and  $\|u\|_\H:=\sqrt{\langle u, \H u\rangle}=\|\H^{\frac{1}{2}}u\|$. Moreover, for any set $U\subseteq \U$, we define $\dist(u,U):=\inf_{u'\in U}\|u-u'\|$ and denote the relative interior of $U$ by $\ri (U)$.
For any $u,u',v,v'\in\U$, we have
\[
\label{eq:triangle}
\begin{array}{l}
\langle u,v\rangle_\H=\frac{1}{2}\left(\|u\|_\H^2+\|v\|_\H^2-\|u-v\|_\H^2\right)=\frac{1}{2}\left(\|u+v\|_\H^2-\|u\|_\H^2-\|v\|_\H^2\right).
\end{array}
\]
Let $\theta:\U\to(-\infty,+\infty]$ be an arbitrary closed proper convex function.
We use $\dom\theta$ to denote its effective domain and $\partial\theta$ to denote its subdifferential mapping.
The proximal mapping of $\theta$ associated with $\H\succ 0$ is defined by
$\prox_\H^\theta(u):=\argmin_{v\in\U}\big\{\theta(v)+\frac{1}{2}\|v-u\|_\H^2\big\}$, $\forall u\in\U$.
It holds \cite{lemarechal} that
\[
\label{ineq:proximalmapping}
\begin{array}{c}
\big\|\prox_\H^\theta(v)-\prox_\H^\theta(v')\big\|_\H^2
\le
\big\langle v-v',\prox_\H^\theta(v)-\prox_\H^\theta(v')\big\rangle_\H.
\end{array}
\]

We say that the Slater  constraint qualification (CQ) holds for problem \eqref{problem:primal} if it holds that
$$
\bigl\{(x,y)\ |\ x\in\ri(\dom p),\ y\in\ri(\dom q),\ \A^{*}x+\B^{*}y=c\bigr\}\ne \emptyset.
$$
When the Slater  CQ holds, we know from \cite[Corollaries 28.2.2 \& 28.3.1]{rocbook} that $(\bar x, \bar y)$ is a solution to \eqref{problem:primal} if and only if there is a Lagrangian multiplier $\bar z\in\mathcal Z$ such that
$(\bar x, \bar y, \bar z)$ is a solution to the following Karush-Kuhn-Tucker (KKT) system
\[
\label{eq:kkt}
0\in\partial p(x)+\nabla f(x)+\A z,
\quad
0\in\partial q(y)+\nabla g(y)+\B z,
\quad
\A^* x+\B^* y=c.
\]
If $(\bar x,\bar y,\bar z)\in\X\times\Y\times\Z$ satisfies \eqref{eq:kkt}, from \cite[Corollary 30.5.1]{rocbook} we know that $(\bar x, \bar y)$ is an optimal solution to problem  \eqref{problem:primal} and  $\bar z$ is an optimal solution to the dual of this problem.
To simplify the notation, we denote $w:=(x,y,z)$ and $\W:=\X\times\Y\times\Z$. The solution set of the KKT system \eqref{eq:kkt} for problem \eqref{problem:primal} is denoted by $\overline\W$.

\section{An Inexact Majorized sPADMM}
\label{imspadmm}
Since the two convex functions $f$ and $g$ in problem \eqref{problem:primal} are assumed to be continuously differentiable with Lipschitz continuous gradients, there exist two self-adjoint positive semidefinite linear operators $\hat\Sigma_{f}:\X\to\X$ and $\hat\Sigma_{g}:\Y\to\Y$ such that for $x,x'\in\X$ and $y,y'\in\Y$,
\begin{equation}
\begin{array}{rl}
\label{ineq:majorize}
f(x) &\le \hat{f}(x;x^\prime):=
f({x}^\prime)+\langle \nabla f({x}^\prime),x-x^\prime\rangle
+\frac{1}{2}\|x-x^\prime\|_{\hat{\Sigma}_f}^2,
\\[1mm]
g(y)& \le \hat{g}(y;y^\prime):=
g({y}^\prime)+\langle\nabla g({y}^\prime),y-y^\prime\rangle
+\frac{1}{2}\|y-{y}^\prime\|_{\hat{\Sigma}_g}^2.
\end{array}
\end{equation}
Let $\sigma>0$. The majorized augmented  Lagrangian function of problem \eqref{problem:primal} is defined by, for any $(x^{\prime},y^{\prime})\in\X\times\Y$ and $(x,y,z)\in\X\times\Y\times\Z$,
$$
\begin{array}{rl}
\hat\L_\sigma(x,y;(z,x^{\prime},y^{\prime})):=& p(x)+\hat f(x;x^{\prime})+q(y)+\hat g(y;y^{\prime})
\\[1mm]
&+\langle z,\A^*x+\B^*y-c\rangle+\frac{\sigma}{2}\|\A^*x+\B^*y-c\|^2.
\end{array}
$$
Let $\S:\X\to\X$ and $\T:\Y\to\Y$ be two self-adjoint positive semidefinite linear operators such that
\[
\label{def:MN}
\M := \hat{\Sigma}_f+\S+\sigma\A\A^{*}\succ 0\quad  \mbox{and}\quad  \N := \hat\Sigma_{g}+\T+\sigma\B\B^{*}\succ 0.
\]
Suppose that $\{w^k :=(x^k,y^k,z^k)\}$ is a sequence in $\X\times\Y\times\Z$.
For convenience, we define the two functions $\psi_{k}:\X\to(-\infty,\infty]$ and $\varphi_{k}:\Y\to(-\infty,\infty]$ by
$$
\begin{array}{l}
\psi_k(x)   :=p(x) + \frac{1}{2}\inprod{x}{\M x}-\inprod{l^k_x}{x},\quad
\varphi_k(y):=q(y) + \frac{1}{2}\inprod{y}{\N y}-\inprod{l^k_y}{y},
\end{array}
$$
where
$$
\begin{array}{l}
-l^k_x := \nabla f({x}^k)+\A z^{k}-\M x^k
+\sigma\A(\A^*x^k+\B^*y^{k}-c),
\\[1mm]
-l_{y}^{k}:=
\nabla g({y}^k)+\B z^{k}- \N y^k+\sigma\B(\A^*x^{k+1}+ \B^* y^k -c).
\end{array}
$$
Now, we are ready  to present our
inexact majorized sPADMM for solving  problem \eqref{problem:primal} and some relevant results.

\medskip
\centerline{
\fbox{
\parbox{0.97\textwidth}{
{\bf Algorithm  imsPADMM}: An inexact  majorized semi-proximal ADMM  for solving problem \eqref{problem:primal}.
\\
Let $\tau\in(0,(1+\sqrt{5})/2)$ be the step-length and $\{\varepsilon_k\}_{k\geq 0}$ be a summable sequence of nonnegative numbers. Choose the linear operators $\S$ and $\T$ such that $\M\succ0$ and $\N\succ 0$ in \eqref{def:MN}. Let $w^0:=(x^0,y^0,z^0)\in\dom p\times\dom q\times\Z$ be the initial point.
For $k=0,1,\ldots$, perform the following steps:
\begin{description}
\item[\bf Step 1.] Compute $x^{k+1}$ and $d_x^k\in\partial \psi_k(x^{k+1})$ such that $\norm{\M^{-\frac{1}{2}}d^k_x} \leq \varepsilon_k$ and
\begin{equation}
 \label{barxplus}
 \begin{array}{ll}
x^{k+1} \approx  \bar{x}^{k+1} :&=
\argmin_{x\in\X}\big\{\hat{\L}_\sigma(x,y^k; w^k)+\frac{1}{2}\norm{x-x^k}_{\S}^2\big\}
\\[0.5mm]
&=\argmin_{x\in\X}\psi_k(x).
\end{array}
\end{equation}
\item[\bf Step 2.] Compute $y^{k+1}$ and $d_y^k\in\partial \varphi_k(y^{k+1})$ such that $\norm{\N^{-\frac{1}{2}}d^k_y} \leq \varepsilon_k$ and
\begin{equation}
\label{baryplus}
\begin{array}{rl}
y^{k+1} \approx \bar{y}^{k+1}:=
&\argmin_{y\in\Y}\big\{
\hat{\L}_\sigma(\bar{x}^{k+1},y; w^k)+\frac{1}{2}\norm{y-y^k}_{\T}^2\big\}
\\[1mm]
=&\argmin_{y\in\Y}\big\{ \varphi_k(y)
+ \inprod{\sigma\B\A^*(\bar{x}^{k+1}-x^{k+1})}{y}
\big\}.
\end{array}
\end{equation}
\item[\bf Step 3.] Compute $z^{k+1}:=z^k+\tau\sigma(\A^*x^{k+1}+\B^* y^{k+1}-c)$.
\end{description}}}}

\begin{proposition}
\label{prop:error}
Let $\{w^k\}$ be the sequence generated by the imsPADMM, and $\{\bar x^{k}\}$, $\{\bar y^{k}\}$ be the sequence defined by \eqref{barxplus} and \eqref{baryplus}.
Then, for any $k\ge 0$, we have $\norm{x^{k+1}-\bar x^{k+1}}_{\M}
\le\varepsilon_k$
and $\norm{y^{k+1}-\bar y^{k+1}}_{\N}
\le(1+\sigma\|\N^{-\frac{1}{2}}\B\A^{*}\M^{-\frac{1}{2}}\|)\varepsilon_k$, where $\M$ and $\N$ are defined in \eqref{def:MN}.
\end{proposition}
\begin{proof}
Noting that $0\in\partial p(x^{k+1})+\M x^{k+1} - l_{x}^{k} - d_{x}^{k}$ and $\M \succ 0$, we can write
$x^{k+1}=\prox_{\M}^{p}\big(\M^{-1}(l_{x}^{k}+d_{x}^{k})\big).$
Also, we have
$\bar x^{k+1}=\prox_{\M}^{p}\big(\M^{-1}l_{x}^{k}\big)$.
By using \eqref{ineq:proximalmapping} we can get
$$
\|x^{k+1}-\bar x^{k+1}\|^{2}_{\M}
\;\le\;
\inprod{x^{k+1}-\bar x^{k+1}}{d_{x}^{k}}
=
\inprod{\M^{\frac{1}{2}}(x^{k+1}-\bar x^{k+1})}{\M^{-\frac{1}{2}}d_{x}^{k}}.
$$
Thus, by using the Cauchy-Schwarz inequality, we can readily obtain that $\norm{x^{k+1}-\bar x^{k+1}}_{\M}\le \|\M^{-\frac{1}{2}}d_{x}^{k}\|\le \varepsilon_k$.
Similarly, we can obtain that
$$
\begin{array}{lll}
\norm{y^{k+1}-\bar y^{k+1}}^{2}_{\N}
&\le&
\inprod{y^{k+1}-\bar y^{k+1}}{ d_{y}^{k}+\sigma\B\A^*(\bar{x}^{k+1}-x^{k+1})}
\\[1mm]
&=&
\inprod{\N^{\frac{1}{2}}(y^{k+1}-\bar y^{k+1})}
{\N^{-\frac{1}{2}}d_{y}^{k}}
\\[1mm]
&&+
\sigma\big\langle \N^\frac{1}{2}(y^{k+1}-\bar y^{k+1}),\N^{-\frac{1}{2}}\B\A^*(\bar x^{k+1}-x^{k+1})
\big\rangle.
\end{array}
$$
This, together with $\norm{\N^{-\frac{1}{2}}d^k_y} \leq \varepsilon_k$ and $\norm{x^{k+1}-\bar x^{k+1}}_{\M}\le\varepsilon_k$, implies that
$$
\begin{array}{ll}
\norm{y^{k+1}-\bar y^{k+1}}_{\N}
&\le
\norm{\N^{-\frac{1}{2}}d_{y}^{k}}
+\sigma\norm{\N^{-\frac{1}{2}}\B\A^{*}\M^{-\frac{1}{2}}}\norm{\bar x^{k+1}-x^{k+1}}_\M
\\[1mm]
&\le
\varepsilon_k+\sigma\|\N^{-\frac{1}{2}}\B\A^{*}
\M^{-\frac{1}{2}}\|\varepsilon_k,
\end{array}
$$
and this completes the proof.
\qed
\end{proof}

\section{An imsPADMM with Symmetric Gauss-Seidel Iteration}
\label{sgs-imspadmm}

We first present some results on the one cycle inexact symmetric Gauss-Seidel (sGS) iteration technique introduced in \cite{lisgs}.
Let $s\ge 2$ be a given integer and $\U:=\U_{1}\times\cdots\times\U_{s}$ with each $\U_{i}$ being a finite dimensional real Euclidean space.
For any $u\in\U$ we write $u\equiv(u_{1},\ldots,u_{s})$.
Let $\H:\U\to\U$ be a given self-adjoint positive semidefinite linear operator with the following block decomposition:
\[
\label{eq-Hu}
\H u:=
\begin{pmatrix}
\H_{11}&\H_{12}&\cdots &\H_{1s}\\
\H^*_{12}&\H_{22}&\cdots &\H_{2s}\\
\vdots&\vdots&\ddots&\vdots \\
\H^*_{1s}&\H^*_{2s}&\cdots &\H_{ss}
\end{pmatrix}
\begin{pmatrix}
u_{1}\\
u_{2}\\
\vdots\\
u_{s}
\end{pmatrix},
\quad
\H_{u} :=
\begin{pmatrix}
0&\H_{12}&\cdots&\H_{1s}\\
&\ddots&\ddots&\vdots\\
&&\ddots&\H_{(s-1)s}\\[3pt]
&&&0
\end{pmatrix}
\]
where $\H_{ii}$ are self-adjoint positive definite linear operators, $\H_{ij}:\U_{j}\to\U_{i}$,
$1\leq i < j \leq s$ are linear maps.
We also  define $\H_{d}u:=(\H_{11}u_{1},\ldots,\H_{ss}u_{s})$.
Note that $\H=\H_{d}+\H_{u}+\H_{u}^{*}$ and $\H_{d}$ is positive definite. To simplify later discussions, for any $u\in\U$, we denote
$u_{\le i}:=\{u_1,\ldots,u_i\}$, $u_{\ge i}:=\{u_i,\ldots,u_s\}$, $i=1,\ldots,s$.
We also define the self-adjoint positive semidefinite linear operator ${\rm sGS}(\cH):\U\to\U$ by
\[
\label{eq-sGS}
{\rm sGS}(\cH) :=\H_{u}\H_{d}^{-1}\H_{u}^{*}.
\]
Let $\theta:\U_{1}\to(-\infty,\infty]$ be a given closed proper convex function and $b\in \U$ be a given vector. Consider the quadratic function $h:\U\to(-\infty,\infty)$ defined by
$h(u):=\frac{1}{2}\langle u,\H u\rangle-\langle b,u\rangle$, $\forall u\in\U$.
Let $\tilde{\delta}_i,\delta_i\in\U_i$, $i=1,\ldots,s$ be given error tolerance vectors with
$\tilde{\delta}_1=\delta_{1}$. Define
\begin{eqnarray}
d(\tilde\delta,\delta)
:=\delta +
\H_{u}\H_{d}^{-1}
(\delta- \tilde{\delta}).
\label{def:d-general}
\end{eqnarray}
Suppose that $u^{-}\in\U$ is a given vector. We want to compute
\[
\label{wplus}
u^+:=\argmin_{u\in\U} \big\{
\theta(u_1)
+h(u)
+\frac{1}{2}\|u-u^{-}\|_{{\rm sGS}(\cH)}^2
-\inprod{d(\tilde\delta,\delta)}{u} \big\}.
\]
We have the following result, established by Li, Sun and Toh in \cite{lisgs} to generalize and reformulate their Schur complement based decomposition method
in \cite{li14} to the inexact setting,
for providing an equivalent  implementable procedure for computing $u^{+}$. This result is essential for our subsequent algorithmic developments.
\begin{proposition}[sGS decomposition]
\label{prop:sgs}
Assume that $\H_{ii}, i=1,\ldots,s$ are positive definite. Then
$$
\hat\H:=\H+{\rm sGS}(\H) =(\H_{d}+\H_{u})\H_{d}^{-1}(\H_{d}+\H_{u}^{*})\succ 0.
$$
Furthermore, for $i=s,s-1,\dots,2$, define $\tilde u_{i}$ by
\[
\label{eq:scbsub0}
\begin{array}{lll}
\widetilde u_i
&:= &\argmin_{u_i}
\big\{\theta(u^{-}_1)
+h(u^{-}_{\le i-1},u_i,\widetilde u_{\ge i+1})
-\langle\widetilde\delta_i,u_i\rangle
\big\}.
\end{array}
\]
Then the optimal solution $u^+$ defined by \eqref{wplus} can be obtained exactly via
\[
\label{eq:scbsub}
\left\{
\begin{array}{lll}
u_1^+
&:=&\argmin_{u_1}\big\{\theta(u_1)+h(u_1,\widetilde u_{\ge2})-\langle\delta_1,u_1\rangle\big\},
\\[1mm]
u_i^{+}
&:=&\argmin_{u_i}
\big\{
\theta(u_1^+)
+h(u^+_{\le i-1},u_i,\widetilde u_{\ge i+1})-\langle\delta_i,u_i\rangle\big\},
\quad i=2,\ldots,s.
\end{array}
\right.
\]
Moreover, the vector $d(\tilde{\delta},\delta)$ defined in \eqref{def:d-general} satisfies
\[
\label{ineq:sgserror}
\|\hat\H^{-\frac{1}{2}}d(\tilde{\delta},\delta)\|
\le \|\H_{d}^{-\frac{1}{2}}(\delta-\tilde\delta)\|
+\|\H_{d}^{\frac{1}{2}}(\H_{d}+\H_{u})^{-1}\tilde\delta\|.
\]
\end{proposition}

We should note that the above sGS decomposition theorem  is valid only
when the (possibly nonsmooth) function $\theta(\cdot)$ is dependent solely on
the first block variable $u_1$, and it is not applicable if there is an additional
nonsmooth convex function involving another block of variable.
In the above proposition, one should interpret  $\tilde{u}_i$ in \eqref{eq:scbsub0} and $u^+_i$ in \eqref{eq:scbsub}
as approximate solutions to the minimization problems without the terms involving
$\tilde{\delta}_i$ and $\delta_i$. Once
these approximate solutions have been computed, they would generate the error vectors
$\tilde{\delta}_i$ and $\delta_i$.
With these known error vectors, we know that $\tilde{u}_i$ and $u^+_i$  are actually
the exact solutions to the minimization problems in \eqref{eq:scbsub0} and
\eqref{eq:scbsub}.
It is important for us to emphasize that when solving the subproblems in the
forward GS sweep in \eqref{eq:scbsub} for $i=2,\ldots,s$, we may try to
estimate $u^+_i$ by using $\tilde{u}_i$, and in this case the corresponding error
vector $\delta_i$ would be given by
$\delta_i = \tilde{\delta}_i + \mbox{$\sum_{j=1}^{i-1}$} \H_{ji}^* (u^+_j-u_j^-)$.
In order to avoid solving the $i$-th problem in \eqref{eq:scbsub}, one may accept
such an approximate solution $u^+_i = \tilde{u}_i$ if the corresponding
error vector {satisfies} an admissible condition such as
$\norm{\delta_i} \leq c\norm{\tilde{\delta}_i}$ for some
constant $c>1$, say $c=10$.

We now show how to apply the  sGS iteration technique in Proposition \ref{prop:sgs}
to the imsPADMM proposed in Sec. \ref{imspadmm}. We should note that in the imsPADMM, the main issue is how to choose $\S$ and $\T$, and how to compute $x^{k+1}$ and $y^{k+1}$.
For later discussions, we use the following  decompositions
$$
\begin{pmatrix}
(\hat\Sigma_{f})_{11}&(\hat\Sigma_{f})_{12}&\cdots &(\hat\Sigma_{f})_{1m}\\
(\hat\Sigma_{f})^*_{12}&(\hat\Sigma_{f})_{22}&\cdots &(\hat\Sigma_{f})_{2m}\\
\vdots&\vdots&\ddots&\vdots \\
(\hat\Sigma_{f})^*_{1m}&(\hat\Sigma_{f})^*_{2m}&\cdots &(\hat\Sigma_{f})_{mm}
\end{pmatrix}
\
\mbox{and}\
\begin{pmatrix}
(\hat\Sigma_{g})_{11}&(\hat\Sigma_{g})_{12}&\cdots &(\hat\Sigma_{g})_{1n}\\
(\hat\Sigma_{g})^*_{12}&(\hat\Sigma_{g})_{22}&\cdots &(\hat\Sigma_{g})_{2n}\\
\vdots&\vdots&\ddots&\vdots \\
(\hat\Sigma_{g})^*_{1n}&(\hat\Sigma_{g})^*_{2n}&\cdots &(\hat\Sigma_{g})_{nn}
\end{pmatrix}
$$
for $\hat\Sigma_{f}$ and $\hat\Sigma_{g}$, respectively, which are consistent with the decompositions of $\X$ and $\Y$.

First, We choose two self-adjoint positive semidefinite linear operators $\tilde\S_{1}:\X_{1}\to\X_{1}$ and $\tilde\T_{1}:\Y_{1}\to\Y_{1}$ for the purpose making the minimization subproblems involving $p_{1}$ and $q_{1}$ easier to solve.
We need $\tilde\S_{1}$ and
$\tilde\T_{1} $  to satisfy the conditions that
$\tilde\M_{11}:=\tilde\S_{1}+(\hat\Sigma_{f})_{11}+\sigma\A_{1}\A_{1}^{*}\succ 0$ as well as $\tilde\N_{11}:=\tilde\T_{1}+(\hat\Sigma_{g})_{11}+\sigma\B_{1}\B_{1}^{*}\succ 0$.
With appropriately chosen $\tilde\S_1$ and $\tilde\T_1$, we can assume that the well-defined optimization problems
$$
\begin{array}{l}
{\min_{x_1}} \big\{ p(x_{1})+\frac{1}{2}\|x_{1}-x_{1}'\|^{2}_{\tilde\M_{11}} \big\} \quad\mbox{and}\quad
{\min_{y_1}} \big\{ q(y_{1})+\frac{1}{2}\|y_{1}-y_{1}'\|^{2}_{\tilde\N_{11}} \big\}
\end{array}
$$
can be solved to arbitrary accuracy for any given $x_{1}'\in\X_{1}$ and $y_{1}'\in\Y_{1}$.

Next, for $i=2,\ldots,m$, we choose a
linear operator $\tilde S_{i}\succeq 0$ such that
$
\tilde\M_{ii}:=\tilde\S_{i}+(\hat\Sigma_{f})_{ii}+\sigma\A_{i}\A_{i}^{*}\succ 0
$ and similarly, for $j=2,\ldots,n$, we choose a linear operator $\tilde\T_{j}\succeq 0$ such that $\tilde\N_{jj}:= \tilde\T_{j}+(\hat\Sigma_{g})_{jj}+\sigma\B_{j}\B_{j}^{*}\succ 0$.

Now,  we define the linear operators
\[
\tilde\M:= \hat\Sigma_f + \sigma \A\A^* + {\rm Diag}(\tilde\S_1,\ldots,\tilde\S_m),
\;\;  \tilde\N:= \hat\Sigma_g + \sigma \B\B^* + {\rm Diag}(\tilde\T_1,\ldots,\tilde\T_n).
\label{eq-MN}
\]
Moreover, define $\tilde\M_u$ and $\tilde\N_u$ analogously as
$\H_u$ in \eqref{eq-Hu} for $\tilde\M$ and $\tilde\N$, respectively, and
$$
\tilde\M_{d}:=\diag(\tilde\M_{11},\ldots,\tilde\M_{mm}),
\quad
\tilde\N_{d}:=\diag(\tilde\N_{11},\ldots,\tilde\N_{nn}).
$$
Then, $\tilde\M:=\tilde\M_{d}+\tilde\M_{u}+\tilde\M_{u}^{*}$ and $\tilde\N:=\tilde\N_{d}+\tilde\N_{u}+\tilde\N_{u}^{*}$.
Moreover, we define the following linear operators:
$$
\begin{array}{c}
\hat\S:=\diag(\tilde\S_{1},\dots,\tilde\S_{m})+ {\rm sGS}(\tilde \M),
\quad
\hat\M:=\hat\Sigma_{f}+\sigma\A\A^{*} +\hat\S,
\\[1mm]
\hat\T:=\diag(\tilde\T_{1},\dots,\tilde\T_{n})+ {\rm sGS}(\tilde \N)
\quad
\mbox{and}
\quad
\hat\N:=\hat\Sigma_{g}+\sigma\B\B^{*} +\hat\T,
\end{array}
$$
where ${\rm sGS}(\tilde\M)$ and ${\rm sGS}(\tilde\N)$ are defined as in \eqref{eq-sGS}.
Define the two constants
\[
\label{def:kappa}
\begin{array}{l}
\kappa := 2\sqrt{m-1}\|\tilde\M_{d}^{-\frac{1}{2}}\|+\sqrt{m}\|\tilde\M_{d}^{\frac{1}{2}}(\tilde\M_{d}+\tilde\M_{u})^{-1}\|,
\\[1mm]
\kappa'  := 2\sqrt{n-1}\|\tilde\N_{d}^{-\frac{1}{2}}\|+\sqrt{n}\|\tilde\N_{d}^{\frac{1}{2}}(\tilde\N_{d}+\tilde\N_{u})^{-1}\|.
\end{array}
\]
Based on the above discussions, we are ready to present the sGS-imsPADMM algorithm for solving problem \eqref{problem:primal}.

\centerline{
\fbox{\parbox{0.97\textwidth}{
{\bf Algorithm sGS-imsPADMM:} An inexact sGS based majorized semi-proximal ADMM for solving problem \eqref{problem:primal}.
\\
Let $\tau\in(0,(1+\sqrt{5})/2)$ be the step-length and $\{\tilde\varepsilon_k\}_{k\ge 0}$ be a summable sequence of nonnegative numbers.
Let $(x^0,y^0,z^0)\in\dom p\times\dom q\times\Z$ be the initial point. For $k=0,1,\ldots$, perform the following steps:
\begin{description}
\item[\bf Step 1a. ({Backward GS sweep})] Compute for $i=m,\ldots,2$,
\end{description}
$$
\begin{array}{l}
\tilde x_{i}^{k+1}\approx{\argmin_{x_{i}\in\X_{i}}} \big\{\hat{\L}_\sigma(x^{k}_{\le i-1},x_{i},\tilde x^{k+1}_{\ge i+1},y^{k}; w^k)+\frac{1}{2}\norm{x_{i}-x_{i}^k}_{\tilde S_{i}}^2 \big\},
\\[3mm]
\tilde\delta^k_i\in\partial_{x_{i}}\hat{\L}_\sigma(x^{k}_{\le i-1},\tilde x_{i}^{k+1},\tilde x^{k+1}_{\ge i+1},y^{k};w^k)+\tilde\S_{i}(\tilde x_{i}^{k+1}-x_{i}^{k})\  \mbox{with}\ \norm{\tilde\delta^k_i} \leq \tilde\varepsilon_k.
\end{array}
$$
\begin{description}
\item[\bf Step 1b. ({Forward GS sweep})] Compute for $i=1,\ldots,m$,
\end{description}
$$
\begin{array}{l}
x_{i}^{k+1} \approx {\argmin_{x_{i}\in\X_{i}}} \big\{\hat{\L}_\sigma(x^{k+1}_{\le i-1},x_{i},\tilde x_{\ge i+1},y^{k}; w^k)+\frac{1}{2}\norm{x_{i}-x_{i}^k}_{\tilde S_{i}}^2\big\},
\\[3mm]
\delta^k_i \in \partial_{x_{i}}\hat{\L}_\sigma(x^{k+1}_{\le i-1},x_{i}^{k+1},\tilde x_{\ge i+1},y^{k}; w^k)+\tilde\S_{i}(x_{i}^{k+1}-x_{i}^{k})\ \mbox{with}\ \norm{\delta_{i}^{k}} \leq \tilde\varepsilon_k.
\end{array}
$$
\begin{description}
\item[\bf Step 2a. ({Backward GS sweep})] Compute for $j=n,\ldots,2$,
\end{description}
$$
\begin{array}{l}
\tilde y_{j}^{k+1} \approx {\argmin_{y_{i}\in\Y_{i}}} \big\{\hat{\L}_\sigma(x^{k+1},y_{\le j-1}^{k},y_{j},\tilde y^{k+1}_{\ge j+1}; w^k)+\frac{1}{2}\norm{y_{j}-y_{j}^k}_{\tilde \T_{j}}^2\big\},
\\[3mm]
\tilde\gamma^k_j \in \partial_{y_{j}}\hat{\L}_\sigma(x^{k+1},y_{\le j-1}^{k},\tilde y^{k+1}_{j},\tilde y^{k+1}_{\ge j+1}; w^k)+\tilde\T_{j}(\tilde y^{k+1}_{j}-y_{i}^{k}) \ \mbox{with}\ \norm{\tilde\gamma^k_j} \leq \tilde\varepsilon_k.
\end{array}
$$
\begin{description}
\item[\bf Step 2b. ({Forward GS sweep})] Compute for $j=1,\ldots,n$,
\end{description}
$$
\begin{array}{l}
y_{j}^{k+1} \approx{\argmin_{y_{i}\in\Y_{i}}}\big\{\hat{\L}_\sigma(x^{k+1},y_{\le j-1}^{k+1},y_{j},\tilde y^{k+1}_{\ge j+1}; w^k)+\frac{1}{2}\norm{y_{j}-y_{j}^k}_{\tilde \T_{j}}^2\big\},
\\[3mm]
\gamma^k_j \in \partial_{y_{j}}\hat{\L}_\sigma(x^{k+1},y_{\le j-1}^{k+1},y^{k+1}_{j},\tilde y^{k+1}_{\ge j+1}; w^k)+\tilde\T_{j}(y^{k+1}_{j}-y_{i}^{k}) \  \mbox{with}\  \norm{\gamma^k_j} \leq\tilde\varepsilon_k.
\end{array}
$$
\begin{description}
\item[\bf Step 3.]  Compute $z^{k+1}:=z^k+\tau\sigma(\A^*x^{k+1}+\B^* y^{k+1}-c).$
\end{description}
}}}
\medskip
For any $k\ge 0$, and $\tilde\delta^{k}=(\tilde\delta_1^{k},\ldots,\tilde\delta_{m}^{k})$,
$\delta^{k}=(\delta_1^{k},\ldots,\delta_{m}^{k})$,
$\tilde\gamma^{k}=(\tilde\gamma_1^{k},\ldots,\tilde\gamma_{n}^{k})$
and
$\gamma^{k}=(\gamma_1^{k},\ldots,\gamma_{n}^{k})$ such that
 $\tilde\delta_{1}^{k+1}:=\delta_{1}^{k+1}$ and
$\tilde\gamma_{1}^{k+1}:=\gamma_{1}^{k+1}$, we define
\[\label{def:dxdy}
d_{x}^{k}:=\delta^{k}+\tilde\M_{u}\tilde\M_{d}^{-1}(\delta^{k}-\tilde\delta^{k})
\quad
\mbox{and}
\quad
d_{y}^{k}:=\gamma^{k}+\tilde\N_{u}\tilde\N_{d}^{-1}(\gamma^{k}-\tilde\gamma^{k}).
\]
Then, for the sGS-imsPADMM we have the following result.
\begin{proposition}
\label{prop:subgradient}
Suppose that $\tilde\M_{d}\succ 0$ and $\tilde\N_{d}\succ 0$
for $\tilde\M$ and $\tilde\N$ defined in \eqref{eq-MN}. Let $\kappa$ and $\kappa'$ be defined as in \eqref{def:kappa}.
Then, the sequences $\{w^k:=(x^{k},y^{k},z^{k})\}$, $\{\delta^{k}\}$, $\{\tilde\delta^{k}\}$, $\{\gamma^{k}\}$ and $\{\tilde\gamma^{k}\}$ generated by the sGS-imsPADMM are well-defined and it holds that
\[
\label{mnpd}
\hat\M= \tilde\M+{\rm sGS}(\tilde\M)  \succ 0,
\quad
\hat\N=  \tilde\N + {\rm sGS}(\tilde\N)  \succ 0.
\]
Moreover, for any $k\ge 0$, $d_{x}^{k}$ and $d_{y}^{k}$ defined by \eqref{def:dxdy} satisfy
\begin{eqnarray}
\label{relation:subgradient}
\left\{\begin{array}{l}
d_x^{k}\in\partial_{x}\big(\hat\L^{k}_\sigma(x^{{k+1}},y^{k})+\frac{1}{2}\|x^{k+1}-x^k\|^{2}_{\hat\S}\big),\\[2mm]
d_y^{k}\in\partial_{y}\big(\hat\L^{k}_\sigma(x^{{k+1}},y^{k+1})+\frac{1}{2}\|y^{k+1}-y^k\|^{2}_{\hat\T}\big),
\end{array} \right.
\\[5pt]
\label{ineq:d-ineq}
\begin{array}{l}
\|\hat\M^{-\frac{1}{2}}d_{x}^{k}\|
\le\kappa \tilde \varepsilon_{k},
\quad
\|\hat\N^{-\frac{1}{2}}d_{y}^{k}\|
\le\kappa' \tilde\varepsilon_{k}. \qquad
\end{array}
\end{eqnarray}
\end{proposition}

\begin{proof}
By Proposition \ref{prop:sgs} we can readily get \eqref{mnpd}. Furthermore, by Proposition \ref{prop:sgs} we can also obtain \eqref{relation:subgradient} from \eqref{def:dxdy}.
By using \eqref{ineq:sgserror} we can get
$$
\begin{array}{ll}
\|\hat\M^{-\frac{1}{2}}d_{x}^{k}\|
&
\le
\|\tilde\M_{d}^{-\frac{1}{2}}\|\|\delta^{k}-\tilde\delta^{k}\|
+\|\tilde\M_{d}^{\frac{1}{2}}(\tilde\M_{d}+\tilde\M_{u})^{-1}\|\|\tilde\delta^{k}\|
\\[1mm]
&\le \big(2\sqrt{m-1}\|\tilde\M_{d}^{-\frac{1}{2}}\|+\sqrt{m}\|\tilde\M_{d}^{\frac{1}{2}}(\tilde\M_{d}+\tilde\M_{u})^{-1}\|\big)\tilde\varepsilon_{k}.
\end{array}
$$
From here and \eqref{def:kappa}, the required inequality for $\|\hat\M^{-\frac{1}{2}}d_{x}^{k}\|$ in \eqref{ineq:d-ineq}  follows.
We can prove the second inequality  in \eqref{ineq:d-ineq} similarly.
\qed
\end{proof}

\begin{remark}
\label{remark4}
(a) If in the imsPADMM, we choose $\S:=\hat\S$, $\T:=\hat\T$, then we have
$\M=\hat\M\succ 0$ and $\N=\hat\N\succ 0$.
Moreover, we can define the sequence $\{\varepsilon_{k}\}$ by $\varepsilon_{k}:=\max\{\kappa,\kappa'\}\tilde\varepsilon_{k}\ \forall k\ge 0$, so that the sequence  $\{\varepsilon_k\}$ is summable if $\{\tilde\varepsilon_k\}$ is summable. Note that  the sequence $\{w^{k}\}$ generated by the sGS-imsPADMM always satisfies $\norm{\M^{-\frac{1}{2}}d^k_x} \leq \varepsilon_k$ and $\norm{\N^{-\frac{1}{2}}d^k_y} \leq \varepsilon_k$. Thus, $\{w^k\}$ can be viewed as a sequence generated by  the imsPADMM  with specially constructed semi-proximal terms.
To sum up, the sGS-imsPADMM is an explicitly implementable method to handle high-dimensional convex composite conic optimization problems, while the imsPADMM has a compact formulation which can facilitate the convergence analysis of the sGS-imsPADMM.
\\[5pt]
(b) As was discussed in the paragraph ensuing Proposition \ref{prop:sgs}, when implementing the sGS-imsPADMM algorithm, we
can use the  $\tilde{x}^{k+1}_i$ computed in the backward GS sweep (Step 1a) to estimate $x^{k+1}_i$  in the forward sweep (Step 1b) for $i=2,\ldots,m$. In this case, the corresponding
error vector is given by
$\delta^k_i = \tilde{\delta}^k_i + \sum_{j=1}^{i-1} \tilde{\M}_{ij}({x}^{k+1}_j-x^k_j)$,
and we may accept the approximate solution $x^{k+1}_i = \tilde{x}^{k+1}_i$ without solving an additional subproblem if $\norm{\delta^k_i}\leq \tilde{\varepsilon}_k$.
A similar strategy also applies to the subproblems in Step 2b for $j=2,\ldots,n$.
\end{remark}

\section{Convergence Analysis}
\label{convergence}

First, we prepare some definitions and notations that will be used throughout this and the next sections.
Since $\{\varepsilon_k\}$ is nonnegative and summable, we can define $\E:=\sum_{k=0}^{\infty}\varepsilon_{k}$ and $\E':=\sum_{k=0}^{\infty}\varepsilon^{2}_{k}$.
Let $\{w^k:=(x^{k},y^{k},z^{k})\}$ be the sequence generated by the imsPADMM and $\{(\bar x^{k},\bar y^{k})\}$ be defined by \eqref{barxplus} and \eqref{baryplus}.
We define the mapping $\R:\X\times\Y\to\Z$ by $\R(x,y):=\A^*x+\B^*y-c$, $\forall(x,y)\in\X\times\Y$, and the following variables, for $k\ge 0$,
$$
\begin{array}{l}
\Delta_x^k:=x^k-x^{k+1},\
\Delta_y^k:=y^k-y^{k+1},\
r^{k}:=\R(x^{k},y^k),\
\bar r^{k}:=\R(\bar x^{k},\bar y^{k}),
\\[1mm]
\Lambda_{x}^{k}:=\bar x^{k}-x^{k},\
\Lambda_{y}^{k}:=\bar y^{k}-y^{k},\
\Lambda_{z}^{k}:=\bar z^{k}-z^{k},
\\[1mm]
\tilde z^{k+1}:=z^{k}+\sigma r^{k+1},\
\bar z^{k+1}:=z^{k}+\tau\sigma\bar r^{k+1},
\end{array}
$$
with the convention that $\bar x^0=x^0$ and $\bar y^0=y^0$.
For any $k\ge 1$, by Clarke's Mean Value Theorem \cite[Proposition 2.6.5]{clarke} there are two self-adjoint linear operators $0\preceq \P_{x}^{k}\preceq\hat\Sigma_{f}$ and $0\preceq\P_{y}^{k}\preceq \hat\Sigma_{g}$ such that
\[
\label{eq:meanvalue-y}
\nabla f(x^{k-1})-\nabla f(x^{k})=\P_{x}^{k}\Delta_x^{k-1}, \quad
\nabla g(y^{k-1})-\nabla g(y^{k})=\P_{y}^{k}\Delta_y^{k-1}.
\]
We define three constants $\alpha$, $\hat\alpha$, $\beta$ by $\alpha:=(1+{\tau}/{\min\{1+\tau,1+\tau^{-1}\}})/2$,
\[
\label{def:beta}
\hat\alpha:=1-\alpha\min\{\tau,\tau^{-1}\},
\quad
\beta:=\min\{1,1-\tau+\tau^{-1}\}\alpha-(1-\alpha)\tau.
\]
Since $\tau\in(0,(1+\sqrt{5})/2)$, it holds that $0<\alpha<1$, $0<\hat\alpha<1$ and $\beta>0$.
Now, we define for any $w\in\W$ and $k\ge 0$,
\begin{eqnarray}
&&\begin{array}{ll}
\phi_{k}(w):=&\
\frac{1}{\tau\sigma}\|z-z^{k}\|^{2}
+\|x-x^{k}\|_{\hat\Sigma_{f}+\S}^{2}
+\|y-y^{k}\|_{\hat\Sigma_{g}+\T}^{2}
\\[1mm]
&+\;\sigma\|\R(x,y^{k})\|^{2}
+\hat\alpha\sigma\norm{r^k}^{2}
+\alpha\norm{\Delta_y^{k-1}}_{\hat\Sigma_g+\T}^2\, ,
\end{array}
\label{def:phik}
\\[1mm]
&&\begin{array}{ll}
\bar\phi_{k}(w):=&
\frac{1}{\tau\sigma}\|z-\bar z^{k}\|^{2}
+\|x-\bar x^{k}\|_{\hat\Sigma_{f}+\S}^{2}
+\|y-\bar y^{k}\|_{\hat\Sigma_{g}+\T}^{2}
\\[1mm]
\nonumber
&+\;\sigma\|\R(x,\bar y^{k})\|^{2}
+\hat\alpha\sigma\norm{\bar{r}^k}^{2}
+\alpha\norm{\bar y^{k}-y^{k-1}}_{\hat\Sigma_g+\T}^2\, .
\end{array}
\end{eqnarray}
Moreover, since $f$ and $g$ are continuously differentiable, there exist
 two self-adjoint positive semidefinite linear operators $\Sigma_f\preceq\hat\Sigma_f$ and $\Sigma_g\preceq\hat\Sigma_g$ such that
\begin{equation}
\left\{
\begin{array}{rl}
f(x)&\ge f({x}^\prime)+\langle \nabla f({x}^\prime),x-{x}^\prime\rangle
+\frac{1}{2}\|x-{x}^\prime\|_{\Sigma_f}^2,
\ \forall\, x,x'\in\X,
\label{ineq:convex-fg}
\\[1mm]
g(y)&\ge g({y}^\prime)+\langle \nabla g({y}^\prime), y- {y}^\prime\rangle
+\frac{1}{2}\|y-{y}^\prime\|_{\Sigma_g}^2, \ \forall\, y,y'\in\Y.
\end{array}
\right.
\end{equation}
Additionally, we define the following two  linear operators
\[
\label{def:FG}
\begin{array}{l}
\F:=\frac{1}{2}\Sigma_{f}+\S+\frac{(1-\alpha)\sigma}{2}\A\A^{*}, \quad
\G:=\frac{1}{2}\Sigma_{g}+\T+\min\{\tau,1+\tau-\tau^2\}\alpha\sigma\B\B^{*}.
\end{array}
\]
The following lemma will be used later.
\begin{lemma}
\label{lemma:sq-sum}
Let $\{a_k\}_{k\ge0}$ be a nonnegative sequence satisfying   $a_{k+1}\le a_k+\varepsilon_k$ for all $k\ge 0$,  where $\{\varepsilon_k\}_{k\ge 0}$ is  a nonnegative and summable sequence of real numbers.
Then the  quasi-Fej\'er monotone sequence $\{a_k\}$ converges to a unique limit point.
\end{lemma}

Now we start to analyze the convergence of the imsPADMM.
\begin{lemma}
\label{lemma:4.1}
Let $\{w^k\}$ be the sequence generated by the imsPADMM for solving problem \eqref{problem:primal}. For any $k\ge1$, we have
\begin{eqnarray}
&&
\begin{array}{l}
(1-\tau)\sigma\|r^{k+1}\|^2
+\sigma\|\R(x^{k+1},y^k)\|^2+2\alpha\inprod{d_y^{k-1}-d_y^k}{\Delta_y^k}
\\[1mm]
\ge \;
\hat{\alpha}\sigma(\| r^{k+1}\|^{2}-\|r^{k}\|^{2})
+\beta \sigma\norm{r^{k+1}}^{2}
+\norm{\Delta_x^{k}}_{\frac{(1-\alpha)\sigma}{2}\A\A^{*}}^{2}
\\[1mm]
\quad
-\|\Delta_y^{k-1}\|_{\alpha(\hat\Sigma_g+\T)}^2
+\|\Delta_y^k\|_{\alpha(\hat\Sigma_g+\T)+\min\{\tau,1+\tau-\tau^2\}\alpha\sigma\B\B^{*}}^2.
\end{array}
\label{ineq:rr}
\end{eqnarray}
\end{lemma}
\begin{proof}
First note that
$\sigma r^{k+1}=\widetilde z^{k+1}-\widetilde z^k+(1-\tau)\sigma r^k$.
By using the fact that $\R(x^{k+1},y^k)=r^{k+1}+\B^*\Delta_y^k$, we have
\begin{eqnarray}
\label{eq:rkplus1}
\begin{array}{l}
(1-\tau)\sigma\|r^{k+1}\|^2
+\sigma\|\R(x^{k+1},y^k)\|^2
\\[1mm]
=(2-\tau)\sigma\|r^{k+1}\|^2
+\sigma\|\B^*\Delta_y^k\|^2
+2\langle \sigma r^{k+1},\B^*\Delta_y^k\rangle
\\[1mm]
=(2-\tau)\sigma\|r^{k+1}\|^2
+\sigma\|\Delta_y^k\|^2_{\B\B^*}
+2(1-\tau)\sigma\langle r^k,\B^{*}\Delta_y^k\rangle
+2\langle \widetilde z^{k+1}-\widetilde z^k,\B^*\Delta_y^k\rangle.
\end{array}
\end{eqnarray}
Since $\hat\Sigma_{g}\succeq \P_{y}^{k}\succeq 0$,
by using \eqref{eq:triangle}
with $\cH := \hat\Sigma_g+\T-\P_y^k\succeq 0$ and $u=\Delta_y^{k-1}$, $v=\Delta_y^{k}$,
we have
\[
\label{ineq:yy}
\begin{array}{ll}
-2\inprod{\Delta_y^{k-1}}{\Delta_y^k}_{\hat\Sigma_g+\T-\P_{y}^k}
\;\geq\;
-\norm{\Delta_y^k}_{\hat\Sigma_g+\T-\P_y^k}^2
-\norm{\Delta_y^{k-1}}_{\hat\Sigma_g+\T-\P_y^k}^2
\\[2mm]
\ge
-\norm{\Delta_y^k}_{\hat\Sigma_g+\T}^2
-\norm{\Delta_y^{k-1}}_{\hat\Sigma_g+\T}^2.
\end{array}
\]
From Step 2 of the imsPADMM we know that for any $k\ge 0$,
\[
\label{opt:y}
\begin{array}{l}
d_{y}^k-\nabla g(y^k)-\B\widetilde z^{k+1}+(\hat\Sigma_g+\T)\Delta_y^k\in\partial q(y^{k+1}).
\end{array}
\]
By using \eqref{opt:y} twice and the maximal monotonicity of $\partial q$, we have for $k\ge 1$,
$$
\begin{array}{l}
\inprod{
 d_y^k-d_y^{k-1}+\nabla g(y^{k-1})-\nabla g(y^k)-\B(\widetilde z^{k+1}-\widetilde z^k)}{-\Delta_y^k}
\\[1mm]
+\inprod{(\hat\Sigma_g+\T)(\Delta_y^k-\Delta_y^{k-1})}{-\Delta_y^k} \ge 0,
\end{array}
$$
which, together with \eqref{eq:meanvalue-y} and \eqref{ineq:yy}, implies that
\begin{eqnarray}
\label{ineq:zzbd}
\begin{array}{l}
\inprod{\widetilde z^{k+1}-\widetilde z^k}{\B^*\Delta_y^{k}}
-\inprod{d_y^k- d_y^{k-1}}{\Delta_y^k}
\\[1mm]
\ge\|\Delta_y^k\|_{\hat\Sigma_g+\T}^2
-\inprod{\Delta_y^{k-1}}{\Delta_y^k}_{\hat\Sigma_g+\T-\P_y^k}
\ge
\frac{1}{2}\norm{\Delta_y^k}_{\hat\Sigma_g+\T}^2
-\frac{1}{2}\norm{\Delta_y^{k-1}}_{\hat\Sigma_g+\T}^2
.
\end{array}
\end{eqnarray}
On the other hand, by using the Cauchy-Schwarz inequality we have
\[
\label{ineq:2case}
\tau\|\B^*\Delta_y^k\|^2
+\tau^{-1}\|r^k\|^2
\ge
2\langle r^k,\B^*\Delta_y^k\rangle\\[2mm]
\ge
-\|\B^*\Delta_y^k\|^2
-\|r^k\|^2.
\]
Now, by applying \eqref{ineq:zzbd} and \eqref{ineq:2case} in \eqref{eq:rkplus1}, we can get
\begin{eqnarray}
&& \hspace{-0.5cm} (1-\tau)\sigma\|r^{k+1}\|^2
+\sigma\|\R(x^{k+1},y^k)\|^2
+2\langle d_y^{k-1}-d_y^k,\Delta_y^k\rangle
\nonumber \\
&\ge &
\max\{1-\tau,1-\tau^{-1}\}\sigma(\norm{r^{k+1}}^2-\norm{r^{k}}^2)
+\norm{\Delta_y^{k}}_{\hat\Sigma_g+\T}^2
- \norm{\Delta_y^{k-1}}_{\hat\Sigma_g+\T}^2
\nonumber \\
\quad\quad\
& & +\min\{\tau,1+\tau-\tau^2\}\sigma (\|\B^*\Delta_y^k\|^2+\tau^{-1}\|r^{k+1}\|^2).
\label{ineq:step-length}
\end{eqnarray}
By using the Cauchy-Schwarz inequality we know that
$$
\begin{array}{l}
\|\R(x^{k+1},y^{k})\|^{2}
=\|
r^{k}-\A^{*}\Delta_x^{k}\|^{2}
=\| r^{k}\|^{2}
+\|\A^{*}\Delta_x^{k}\|^{2}
-2\langle r^{k},\A^{*}\Delta_x^{k}\rangle
\\[1mm]
\ge\|r^{k}\|^{2}
+\|\A^{*}\Delta_x^{k}\|^{2}
-2\|r^{k}\|^{2}
-\frac{1}{2}\norm{\A^*\Delta_x^{k}}^{2}
=
\frac{1}{2}\norm{\A^*\Delta_x^{k}}^{2}
-\norm{r^{k}}^2,
\end{array}
$$
so that for any $\alpha\in(0,1]$, we have
\begin{eqnarray}
\label{ineq:1-alpha}
\begin{array}{ll}
&
(1-\alpha)\big((1-\tau)\sigma\|r^{k+1}\|^2
+\sigma\|\R(x^{k+1},y^k)\|^2\big)
\\[1mm]
& \ge
(1-\alpha)\big((1-\tau)\sigma\|r^{k+1}\|^{2}-\sigma\|r^{k}\|^{2}+\frac{\sigma}{2}\|\A^{*}\Delta_x^{k}\|^{2}\big)
\\[1mm]
& =
(\alpha-1)\tau\sigma\|r^{k+1}\|^{2}+(1-\alpha)\sigma(\|r^{k+1}\|^{2}-\|r^{k}\|^{2})
+\frac{(1-\alpha)\sigma}{2}\|\Delta_x^{k}\|_{\A\A^{*}}^{2}.
\end{array}
\end{eqnarray}
Finally by adding \eqref{ineq:1-alpha} to the inequality  generated by multiplying $\alpha$ to both sides of \eqref{ineq:step-length}, we can get \eqref{ineq:rr}.
This completes the proof.
\qed
\end{proof}
Next, we shall derive an inequality which is essential for
establishing both the global convergence and the iteration complexity of  the imsPADMM.

\begin{proposition}
\label{prop:inequality}
Suppose that the solution set $\overline\W$ to the KKT system \eqref{eq:kkt} of problem \eqref{problem:primal} is nonempty. Let $\{w^k\}$ be the sequence generated by  the imsPADMM. Then, for any $\bar w:=(\bar x,\bar y,\bar z)\in\overline\W$ and $k\ge 1$,
\[
\label{ineq:mainb}
\begin{array}{l}
2\alpha\inprod{ d_y^k- d_y^{k-1}}{\Delta_y^k}
-2\inprod{d^k_x}{x^{k+1}-\bar{x}} -  2\inprod{d^k_y}{y^{k+1}-\bar{y}}
\\[2mm]
+\norm{\Delta_x^{k}}^{2}_{\F}+\norm{\Delta_y^{k}}^{2}_{\G}
+\beta\sigma\norm{r^{k+1}}^{2}
\le
\phi_k(\bar w) - \phi_{k+1}(\bar w).
\end{array}
\]
\end{proposition}
\begin{proof}
For any given $(x,y,z)\in\W$, we define $x_e:=x-\bar x$, $y_e:=y-\bar y$ and $z_e:=z-\bar z$.
Note that
$$z^{k}+\sigma\R(x^{k+1},y^{k})
=\tilde z^{k+1}+\sigma\B^{*}(y^{k}-y^{k+1}).$$
Then, from Step 1 of the imsPADMM, we know that
\[
\label{inclu:subx}
d_{x}^{k}-\nabla f(x^{k})
-\A(\tilde z^{k+1}+\sigma\B^{*}\Delta_y^{k})
+(\hat\Sigma_{f}+\S)\Delta_x^{k}
\in
\partial p(x^{k+1}).
\]
Now the convexity of $p$ implies that
\[
\label{ineq:p(x)}
\begin{array}{l}
p(\bar x)
+\inprod{d_{x}^{k}-\nabla f(x^{k})
-\A(\tilde z^{k+1}+\sigma\B^{*}\Delta_y^{k})
+(\hat\Sigma_{f}+\S)\Delta_x^{k}}{x_e^{k+1}}
\ge
p(x^{k+1}).
\end{array}
\]
On the other hand, by using  \eqref{ineq:majorize} and \eqref{ineq:convex-fg}, we  have
$$
\begin{array}{l}
f(\bar x)-f({x}^k) + \inprod{ \nabla f({x}^k)}{{x}_e^k}
\geq
\frac{1}{2}\norm{{x}_e^k}_{\Sigma_f}^2,
\\[1mm]
f({x}^k) - f(x^{k+1}) -\inprod{ \nabla f({x}^k)}{\Delta_x^k}\geq
- \frac{1}{2}\norm{\Delta_x^k}_{\hat{\Sigma}_f}^2.
\end{array}
$$
Thus, summing up the above two inequalities together gives
\[
\label{ineq:f(x)}
\begin{array}{l}
f(\bar x)-f(x^{k+1})
+\inprod{\nabla f(x^k)}{x_e^{k+1}}
\ge
\frac{1}{2}\norm{x_e^{k}}^{2}_{\Sigma_{f}}
-\frac{1}{2}\norm{\Delta_x^{k}}_{\hat\Sigma_{f}}^{2}.
\end{array}
\]
By summing \eqref{ineq:p(x)} and \eqref{ineq:f(x)} together, we get
\[
\label{ineq:p(x)+f(x)}
\begin{array}{l}
p(\bar x)+f(\bar x)-p(x^{k+1})-f(x^{k+1})
-\inprod{\tilde z^{k+1}+\sigma\B^{*}\Delta_y^{k}}{\A^{*}x_e^{k+1}}
\\[1mm]
+\inprod{(\hat\Sigma_{f}+\S)\Delta_x^{k}}{x_e^{k+1}}
+\inprod{ d_{x}^{k}}{x_e^{k+1}}
\ge\frac{1}{2}(\norm{x_e^k}^{2}_{\Sigma_{f}}-\norm{\Delta_x^k}_{\hat\Sigma_{f}}^{2}).
\end{array}
\]
Applying a similar derivation, we also get that for any $ y\in\Y$,
\[
\label{ineq:q(y)+g(y)}
\begin{array}{l}
q(\bar y)+g(\bar y) -q(y^{k+1}) -g(y^{k+1})
-\inprod{\tilde z^{k+1}}{\B^{*}y_e^{k+1}}
\\[1mm]
+\inprod{(\hat\Sigma_{g}+\T)\Delta_y^{k}}{y_e^{k+1}}
+\inprod{ d_{y}^k}{y_e^{k+1}}
\;\ge\;
\frac{1}{2}\big(\norm{y_e^{k}}^{2}_{\Sigma_g}
-\norm{\Delta_y^{k}}_{\hat\Sigma_g}^{2}\big).
\end{array}
\]
By using  (\ref{eq:kkt}),  \eqref{ineq:convex-fg} and the convexity of the functions $f$, $g$, $p$ and $q$, we have
\begin{eqnarray}
\label{ineq:fgkp1}
\begin{array}{l}
p(x^{k+1})+f(x^{k+1})-p(\bar{x})-f(\bar x)
+\inprod{\A\bar{z}}{{x}^{k+1}_e}
\geq
\frac{1}{2}\norm{x_e^{k+1}}_{\Sigma_f}^2,
\\[1mm]
q(y^{k+1})+g(y^{k+1})-q(\bar y)-g(\bar y)
+\inprod{\B\bar{z}}{{y}^{k+1}_e}
\geq \frac{1}{2}\norm{y_e^{k+1}}_{\Sigma_g}^2.
\end{array}
\end{eqnarray}
Finally, by summing \eqref{ineq:p(x)+f(x)}, \eqref{ineq:q(y)+g(y)}, \eqref{ineq:fgkp1}  together, we  get
\[
\label{ineq:pq}
\begin{array}{l}
\inprod{ d_{x}^{k}}{x_e^{k+1}}+\inprod{ d_{y}^k}{y_e^{k+1}}
-\inprod{\tilde z_e^{k+1}}{r^{k+1}}
-\sigma\inprod{\B^{*}\Delta_y^{k}}{\A^*x_e^{k+1}}
\\[1mm]
+\inprod{\Delta_x^{k}}{x_e^{k+1}}_{\hat\Sigma_{f}+\S}
+\inprod{\Delta_y^{k}}{y_e^{k+1}}_{\hat\Sigma_{g}+\T}
+\frac{1}{2}(\norm{\Delta_x^k}_{\hat\Sigma_{f}}^{2}+\norm{\Delta_y^{k}}_{\hat\Sigma_g}^{2})
\\[2mm]
\ge
\frac{1}{2}(\norm{x_e^k}^{2}_{\Sigma_{f}}+\norm{y_e^{k}}^{2}_{\Sigma_g}
+\norm{x_e^{k+1}}_{\Sigma_f}^2
+\norm{y_e^{k+1}}_{\Sigma_g}^2)
\\[2mm]
\ge
\frac{1}{4}\norm{\Delta_x^k}^{2}_{\Sigma_f}
+\frac{1}{4}\norm{\Delta_y^k}^{2}_{\Sigma_g}.
\end{array}
\]
Next, we estimate the left-hand side of
\eqref{ineq:pq}.
By using \eqref{eq:triangle}, we have
\[
\label{eq:AxBy}
\begin{array}{ll}
\inprod{\B^{*}\Delta_y^{k}}{\A^*x_e^{k+1}}
=\inprod{\B^{*}y_e^{k}-\B^{*} y_e^{k+1}}{r^{k+1}-\B^*y_e^{k+1}}
\\[1mm]
=\inprod{\B^{*}y_e^{k}-\B^{*} y_e^{k+1}}{r^{k+1}}
-\frac{1}{2}\big(
\|\B^*y_e^{k}\|^{2}
-\|\B^{*}y_e^{k}-\B^{*} y_e^{k+1}\|^2
-\|\B^{*} y_e^{k+1}\|^2
\big)
\\[1mm]
=\frac{1}{2}\big(\|\B^*y_e^{k+1}\|^{2}
+\|\R(x^{k+1},y^{k})\|^{2}
-\|\B^*y_e^{k}\|^{2}
-\|r^{k+1}\|^{2}\big).
\end{array}
\]
Also, from \eqref{eq:triangle} we know that
\begin{eqnarray}
\label{eq:xxyy1}
\begin{array}{l}
\langle x_e^{k+1},\Delta_x^{k}\rangle_{\hat\Sigma_{f}+\S}
=
\frac{1}{2}(\|x_e^{k}\|_{\hat\Sigma_{f}+\S}^{2}
-\|x_e^{k+1}\|_{\hat\Sigma_{f}+\S}^{2})
-\frac{1}{2}\|\Delta_x^{k}\|_{\hat\Sigma_{f}+\S}^{2},
\\[2mm]
\langle y_e^{k+1},\Delta_y^{k}\rangle_{\hat\Sigma_{g}+\T}
=\frac{1}{2}(
\|y_e^{k}\|_{\hat\Sigma_{g}+\T}^{2}
-\|y_e^{k+1}\|_{\hat\Sigma_{g}+\T}^{2})
-\frac{1}{2}\|\Delta_y^{k}\|_{\hat\Sigma_{g}+\T}^{2}\, .
\end{array}
\end{eqnarray}
Moreover, by using the definition of $\{\tilde z^k\}$ and \eqref{eq:triangle} we know that
\[
\label{eq:rz}
\begin{array}{l}
\langle r^{k+1},\tilde z_e^{k+1}\rangle
=\inprod{r^{k+1}}{z_e^{k}+\sigma r^{k+1}}
=
\frac{1}{\tau\sigma}\langle z^{k+1}-z^{k},z_e^{k}\rangle
+\sigma\|r^{k+1}\|^{2}\\[1mm]
=
\frac{1}{2\tau\sigma}\big(
\|z^{k+1}_e\|^{2}
-\|z^{k+1}-z^{k}\|^{2}
-\|z_e^{k}\|^{2}
\big)
+\sigma\|r^{k+1}\|^{2}\\[1mm]
=
\frac{1}{2\tau\sigma}\big(
\|z^{k+1}_e\|^{2}-\|z_e^{k}\|^{2}
\big)
+\frac{(2-\tau)\sigma}{2}\|r^{k+1}\|^{2}.
\end{array}
\]
Thus, by using \eqref{eq:AxBy}, \eqref{eq:xxyy1} and \eqref{eq:rz}  in \eqref{ineq:pq}, we obtain that
\[
\label{ineq:p(x)+q(y)}
\begin{array}{l}
\inprod{ d_{x}^{k}}{x_e^{k+1}}+\inprod{ d_{y}^k}{y_e^{k+1}}
+\frac{1}{2\tau\sigma}
(\|z_e^{k}\|^{2}-\|z^{k+1}_e\|^{2})
+\frac{\sigma}{2}(\|\B^*y_e^{k}\|^{2}-\|\B^*y_e^{k+1}\|^{2})
\\[2mm]
+\frac{1}{2}
(\|x_e^{k}\|_{\hat\Sigma_{f}+\S}^{2}+\|y_e^{k}\|_{\hat\Sigma_{g}+\T}^{2})
-\frac{1}{2}(\|x_e^{k+1}\|_{\hat\Sigma_{f}+\S}^{2}+\|y_e^{k+1}\|_{\hat\Sigma_{g}+\T}^{2})
\\[2mm]
\ge
\frac{1}{2}\norm{\Delta_x^k}^{2}_{\frac{1}{2}\Sigma_f+\S}
+\frac{1}{2}\norm{\Delta_y^k}^{2}_{\frac{1}{2}\Sigma_g+\T}
+\frac{\sigma}{2}\|\R(x^{k+1},y^{k})\|^{2}
+\frac{(1-\tau)\sigma}{2}\|r^{k+1}\|^{2}.
\end{array}
\]
Note that for any $y\in\Y$, $\R(\bar x,y)=\B^{*}y_{e}$.
Therefore, by applying \eqref{ineq:rr} to the right hand side of \eqref{ineq:p(x)+q(y)} and using \eqref{def:phik} together with \eqref{def:FG}, we know that \eqref{ineq:mainb} holds for $k\ge 1$ and this completes the proof.
\qed
\end{proof}
Now, we are ready to present the convergence theorem of the imsPADMM.
\begin{theorem}
\label{theorem:convergence}
Suppose that the solution set $\overline\W$ to the KKT system \eqref{eq:kkt} of problem \eqref{problem:primal} is nonempty and $\{w^k\}$ is generated by the imsPADMM.
Assume that
\[
\label{ineq:succ0}
\begin{array}{l}
\Sigma_{f}+\S+\sigma\A\A^{*}\succ0
\quad
\mbox{and}
\quad
\Sigma_{g}+\T+\sigma\B\B^{*}\succ0.
\end{array}
\]
Then, the linear operators $\F$ and $\G$ defined in \eqref{def:FG} are positive definite.
Moreover, the sequence $\{w^{k}\}$ converges to a point in $\overline\W$.
\end{theorem}
\begin{proof}
Denote $\rho := \min (\tau, 1+\tau -\tau^2) \in (0,1]$. Since  $\alpha>0$, from  the definitions of $\F$ and $\G$, we know that
$$
\begin{array}{l}
 \F = \frac{(1-\alpha)}{2} \big( \Sigma_f + \S + \sigma\A\A^*\big)
+ \frac{\alpha}{2} \Sigma_f + \frac{1+\alpha}{2} \S  \succ 0,
\\[1mm]
\G = \frac{\rho\alpha}{2}\big( \Sigma_g + \T + \sigma\B\B^*\big) +
\frac{1-\rho\alpha}{2}\Sigma_g + \frac{2-\rho\alpha}{2}\T + \frac{\rho\alpha}{2}\sigma\B\B^* \succ 0 .
\end{array}
$$
Now we start to prove the convergence of the sequence $\{w^{k}\}$.
We first show that  this sequence is bounded.
Let $\bar w:=(\bar x,\bar y,\bar z)$ be an arbitrary vector in $\overline\W$.
For any given $(x,y,z)\in\W$, we define $x_e:=x-\bar x$, $y_e:=y-\bar y$ and $z_e:=z-\bar z$.
Since $\G\succ 0$, it holds that
\[
\label{eq:deltay}
\norm{\Delta_{}y^{k}}^{2}_{\G}
+2\alpha\inprod{ d_y^k- d_y^{k-1}}{\Delta_y^k}
= \norm{\Delta_y^k+\alpha\G^{-1}(d^k_y-d^{k-1}_y)}_\G^2-\alpha^2\norm{d^k_y-d^{k-1}_y}_{\G^{-1}}^2.
\]
By substituting $\bar x^{k+1}$ and $\bar y^{k+1}$ for $x^{k+1}$ and $y^{k+1}$ in \eqref{ineq:mainb}, we obtain that
\begin{equation}
\label{ineq:contraction2}
\begin{array}{l}
\phi_{k}(\bar w)
-\bar \phi_{k+1}(\bar w)
+\alpha^{2}\|d_y^{k-1}\|^{2}_{\G^{-1}}
\\[1mm]
\ge
\|\bar x^{k+1}-x^{k}\|^{2}_{\F}
+\beta\sigma\|\bar r^{k+1}\|^{2}
+\|\bar y^{k+1}-y^{k}+\alpha\G^{-1}d_y^{k-1}\|_{\G}^{2}.
\end{array}
\end{equation}
Define the sequences $\{\xi^{k}\}$ and $\{\bar \xi^{k}\}$
 in $\Z\times\X\times\Y\times\Z\times\Y$
 for $k\ge 1$ by
\[
\label{def:xi}
\left\{
\begin{array}{l}
\xi^{k}:=\big(
\sqrt{\tau\sigma}
z^{k}_{e},
(\hat\Sigma_{f}+\S)^{\frac{1}{2}}x^{k}_{e},
\N^{\frac{1}{2}}y^{k}_{e},
\sqrt{\hat\alpha\sigma}r^{k},
\sqrt{\alpha}(\hat\Sigma_g+\T)^{\frac{1}{2}}(\Delta_y^{k-1})
\big),
\\[1mm]
\bar\xi^{k}:=\big(
\sqrt{\tau\sigma}\bar z^{k}_{e},
(\hat\Sigma_{f}+\S)^{\frac{1}{2}}\bar x^{k}_{e},\N^{\frac{1}{2}}\bar y^{k}_e,
\sqrt{\hat\alpha\sigma}\bar r^{k},
\sqrt{\alpha}(\hat\Sigma_g+\T)^{\frac{1}{2}}(y^{k-1}-\bar y^{k})
\big).
\end{array}
\right.
\]
Obviously we have $\|\xi_{k}\|^{2}=\phi_{k}(\bar w)$ and $\|\bar \xi_{k}\|^{2}=\bar \phi_{k}(\bar w)$, which, together with
\eqref{ineq:contraction2} implies that
$
\|\bar\xi^{k+1}\|^{2}
\le
\|\xi^{k}\|^{2}
+\alpha^{2}\|\G^{-\frac{1}{2}} d_y^{k-1}\|^{2}
$.
As a result, it holds that
$\|\bar\xi^{k+1}\|
\le\|\xi^{k}\|+\alpha\|\G^{-\frac{1}{2}} d_y^{k-1}\|$.
Therefore, we have that
\[
\label{ineq:xik+1}
\|\xi^{k+1}\|
\le
\|\xi^{k}\|
+\alpha\|\G^{-\frac{1}{2}} d_y^{k-1}\|+
\|\bar\xi^{k+1}-\xi^{k+1}\|.
\]
Next, we estimate $\|\bar\xi^{k+1}-\xi^{k+1}\|$ in \eqref{ineq:xik+1}.
Since $\hat{\alpha} + \tau \in [1,2]$, it holds that
$$
\begin{array}{l}
\frac{1}{\tau\sigma}\norm{\Lambda_z^{k+1}}^2
+{\hat\alpha\sigma}\norm{\bar r^{k+1}-r^{k+1}}^2
=(\tau+\hat{\alpha})\sigma\norm{\bar r^{k+1}-r^{k+1}}^2
\\[1mm]
\le
2\sigma
\norm{\A^*\Lambda_x^{k+1}+\B^*\Lambda_y^{k+1}}^2
\le
4\norm{\Lambda_x^{k+1}}_{\sigma\A\A^*}^{2}
+4\norm{\Lambda_y^{k+1}}_{\sigma\B\B^{*}}^2,
\end{array}
$$
which, together with Proposition  \ref{prop:error}, implies that
\[
\label{ineq:xixi}
\begin{array}{l}
\norm{\bar \xi^{k+1}-\xi^{k+1}}^2
\\[1mm]
\le
\norm{\Lambda_x^{k+1}}^2_{\hat\Sigma_{f}+\S}
+\norm{\Lambda_y^{k+1}}^2_{\N}
+\norm{\Lambda_y^{k+1}}^2_{\hat\Sigma_g+\T}
+4\norm{\Lambda_x^{k+1}}_{\sigma\A\A^*}^{2}
+4\norm{\Lambda_y^{k+1}}_{\sigma\B\B^{*}}^2
\\[2mm]
\leq
5(\norm{\Lambda_x^{k+1}}_{\M}^{2}
+\norm{\Lambda_{y}^{k+1}}_{\N}^{2})
\le
\varrho^2 \varepsilon_k^{2},
\end{array}
\]
where $\varrho$ is a constant defined by
\[
\label{def:varrho}
\varrho:=\sqrt{5(1+(1+\sigma\|\N^{-\frac{1}{2}}\B\A^{*}\M^{-\frac{1}{2}}\|)^2)}.
\]
On the other hand, from Proposition  \ref{prop:error}, we know  $\norm{\G^{-\frac{1}{2}} d_y^{k}}\le\norm{\G^{-\frac{1}{2}}\N^{\frac{1}{2}}} \varepsilon_{k}$.
By using this fact together with \eqref{ineq:xik+1} and \eqref{ineq:xixi}, we have
\[
\label{ineq:xibound}
\begin{array}{rl}
\norm{\xi^{k+1}}
\le \norm{\xi^k}
+\varrho\varepsilon_{k}
+\norm{\G^{-\frac{1}{2}}\N^{\frac{1}{2}}}\varepsilon_{k-1}
\le\|\xi^{1}\|+\big(\varrho+ \norm{\G^{-\frac{1}{2}}\N^{\frac{1}{2}}}\big)\E,
\end{array}
\]
which implies that the sequence $\{\xi^{k}\}$ is bounded. Then, by \eqref{ineq:xixi} we know that the sequence $\{\bar \xi^{k}\}$  is also bounded.
From the definition of  $\xi^{k}$ we know that the sequences $\{y^{k}\}$, $\{z^{k}\}$, $\{ r^k\}$ and $\{(\hat\Sigma_{f}+\S)^{\frac{1}{2}}x^{k}\}$ are bounded.
Thus, by the definition of $r^k$, we know that the sequence $\{Ax^k\}$ is also bounded, which  together with the definition of $\M$ and the fact that $\M\succ 0$, implies that $\{x^k\}$ is bounded.

By \eqref{ineq:contraction2}, \eqref{ineq:xixi} and \eqref{ineq:xibound} we have that
\[
\label{ineq:summable}
\begin{array}{l}
\sum_{k=1}^\infty
\big(
\norm{\bar{x}^{k+1}-x^{k}}^{2}_{\F}
+\beta\sigma\norm{\bar r^{k+1}}^{2}
+\norm{\bar{y}^{k+1}-y^{k}+\alpha\G^{-1}d_y^{k-1}}_{\G}^{2}\big)
\\[1mm]
\leq
\sum_{k=1}^\infty \big(\phi_{k}(\bar w) - \phi_{k+1}(\bar w)
+\phi_{k+1}(\bar w) -\bar \phi_{k+1}(\bar w)
+\alpha^{2}\|d_y^{k-1}\|^{2}_{\G^{-1}}\big)
\\[1mm]
\leq
\phi_{1}(\bar w)
+\mbox{$\sum_{k=1}^\infty$} \norm{\xi^{k+1}-\bar{\xi}^{k+1}}
(\norm{\xi^{k+1}}+\norm{\bar{\xi}^{k+1}})
+\norm{\G^{-\frac{1}{2}}\N^{\frac{1}{2}}}^2\E'
\\[1mm]
\leq
\phi_{1}(\bar w)
+\norm{\G^{-\frac{1}{2}}\N^{\frac{1}{2}}}^2\E'
+\varrho \mbox{$\max\limits_{k\geq 1}$}\{ \norm{\xi^{k+1}}+\norm{\bar{\xi}^{k+1}}\}\E
<\infty,
\end{array}
\]
where we have used the fact that $\phi_{k}(\bar w) -\bar \phi_{k}(\bar w)
\leq \norm{\xi^{k}-\bar{\xi}^{k}}(\norm{\xi^{k}}+\norm{\bar{\xi}^{k}})$.
Thus, by \eqref{ineq:summable} we know that
$\{\norm{\bar{x}^{k+1}-x^{k}}^{2}_{\F}\}\to0$,
$\{\norm{\bar{y}^{k+1}-y^{k}+\alpha\G^{-1}d_y^{k-1}}_{\G}^2\}\to0$,
and
$\{\norm{\bar r^{k+1}}^{2}\}\to0$ as $k\to\infty$.
Since we have $\F\succ 0$ and $\G\succ 0$ by \eqref{ineq:succ0}, $\{\bar{x}^{k+1}-x^{k}\}\to 0$, $\{\bar{y}^{k+1}-y^{k}\}\to 0$
and $\{\bar{r}^{k}\}\to 0$ as $k\to\infty$.
Also, since $\M\succ 0$ and $\N\succ 0$, by Proposition \ref{prop:error} we know that
$\{\bar{x}^{k}-x^{k}\}\to 0$ and $\{\bar{y}^{k}-y^{k}\}\to 0$ as $k\to\infty$.
As a result, it holds that
$\{\Delta_x^{k}\}\to0$,
$\{\Delta_y^{k}\}\to 0$,
and $\{r^{k}\}\to 0$ as $k\to\infty$.
Note that the sequence $\{(x^{k+1},y^{k+1},z^{k+1})\}$ is bounded. Thus, it has a convergent subsequence  $\{(x^{k_{i}+1},y^{k_{i}+1},z^{k_{i}+1})\}$ which converges to a point, say
$(x^{\infty},y^{\infty},z^{\infty})$.
We define two nonlinear mappings $F:\W\to\X$ and $G:\W\to\Z$ by
\[
\label{def:mapfg}
F(w):=\partial p(x)+\nabla f(x)+\A z,\  G(w):=\partial q(y)+\nabla g(y)+\B z, \forall w\in\W.
\]
From \eqref{inclu:subx}, \eqref{opt:y} and \eqref{eq:meanvalue-y} we know that in the imsPADMM, for any $k\ge 1$, it holds that
\[
\left\{
\label{opt:2sub}
\begin{array}{l}
d_{x}^{k}-\P_x^{k+1}\Delta_x^k
+(\hat\Sigma_{f}+\S)\Delta_x^k
+(\tau-1)\sigma\A r^{k+1}
-\sigma\A\B^{*}\Delta_y^{k}\in F(w^{k+1}),
\\[1mm]
d_{y}^k
-\P^{k+1}_y\Delta_y^k
+(\hat\Sigma_g+\T)\Delta_y^k
+(\tau-1)\sigma\B r^{k+1}
\in G(w^{k+1}).
\end{array}
\right.
\]
Thus by taking limits along $\{k_i\}$ as $i\to\infty$ in \eqref{opt:2sub}, we know that
$$
\begin{array}{l}
0\in\partial p(x^\infty)+\nabla f(x^\infty)+\A z^\infty,
\quad\mbox{and}\quad
0\in\partial q(y^\infty)+\nabla g(y^\infty)+\B z^\infty,
\end{array}
$$
which together with the fact that $\lim_{k\to\infty}r^k= 0$ implies that $(x^\infty, y^\infty, z^\infty)\in\overline\W$.
Hence, $(x^{\infty},y^{\infty})$ is a solution to the problem \eqref{problem:primal} and $z^{\infty}$ is a solution to the dual of problem \eqref{problem:primal}.

By \eqref{ineq:xibound} and Lemma \ref{lemma:sq-sum}, we know that the sequence $\{\|\xi^k\|\}$ is convergent. We can let $\bar w=(x^\infty, y^\infty, z^\infty)$ in all the previous discussions. Hence, $\lim_{k\to \infty} \xi^k = 0$.
Thus, from the definition of $\{\xi^k\}$ we know that
$\lim_{k\to\infty}z^{k} =z^\infty$,
$\lim_{k\to\infty}y^{k} =y^\infty$
and
$\lim_{k\to\infty} (\hat\Sigma_f+\S)x^{k} =(\hat\Sigma_f+\S)x^\infty$.
Obviously, since $\lim_{k\to \infty} r^k=0$, it holds that $\{\A^*x^k\}\to \A^*x^\infty$ as $k\to\infty$. Finally, we get
$\lim_{k\to\infty}x^{k}=x^{\infty}$ by the definition of $\M$ and the fact that $\M\succ 0$.
This completes the proof.
\qed
\end{proof}

\section{Non-Ergodic  Iteration Complexity}
\label{convergence-rate}
In this section we establish an iteration complexity result in the non-ergodic sense
for the imsPADMM. The definitions and notations in Sec. \ref{convergence} also apply to this section.
We define the function $D:\W\to[0,\infty)$ by
$$
\label{def:D}
D(w):=
{\dist}^{2}(0,F(w))
+
{\dist}^{2}(0,G(w))
+
\|\R(x,y)\|^{2},
\
\forall w=(x,y,z)\in\W,
$$
where $F$ and $G$ are defined in \eqref{def:mapfg}.
We say that ${w}\in\W$ is an $\epsilon$-approximation solution of problem \eqref{problem:primal} if $D({w})\le\epsilon$.
Our iteration complexity result is established based on the KKT optimality condition
in the sense that we can find a point ${w}\in\W$ such that $D({w})\le o(1/k)$
after $k$ steps.
The following lemma will be needed later.

\begin{lemma}
\label{lemma:sequence}
If $\{a_{i}\}$ is a nonnegative sequence  satisfying
$\sum_{i=0}^{\infty}a_{i}=\bar a$,
then we have
$\min\limits_{1\le i\le k}\{a_{i}\}\le \bar a/k$
{and}
$\lim\limits_{k\to\infty}\big\{k\cdot\min\limits_{1\le i\le k}\{a_{i}\}\big\}=0$.
\end{lemma}
Now we establish a non-ergodic iteration complexity for the imsPADMM.
\begin{theorem}
\label{non-ergodic-theorem}
Suppose that all the assumptions of Theorem 1 hold and  the convergent sequence $\{w^k\}$
generated by the imsPADMM converges to the limit $\bar w:=(\bar x,\bar y,\bar z)$.
Then, there exists a constant $\omega>0$  such that
\begin{eqnarray}
\label{ineq:nec-r1}
\min_{1\le i\le k}\big\{D(w^{i+1})\big\}\le \omega/k,\
\mbox{and}
\
\lim_{k\to\infty}\big\{ k\min_{1\le i\le k}\{D(w^{i+1})\}\big\}=0.
\end{eqnarray}
\end{theorem}

\begin{proof}
For any $(x,y,z)\in\X\times\Y\times\Z$, define $x_{e}:=x-\bar x$, $y_{e}:=y-\bar y$ and $z_{e}=z-\bar z$. Moreover, we define the sequence $\{\zeta_{k}\}_{k\ge 1}$  by
$$
\zeta_{k}:=
\mbox{$\sum_{i=1}^{k}$} \big(
2\langle d_x^{i},x^{i+1}_{e}\rangle
+
2\langle d_y^{i},y^{i+1}_{e}\rangle
+
\alpha^{2}\|d_{y}^{i}-d_{y}^{i-1}\|^{2}_{\G^{-1}}\big).
$$
We first show that $\{\zeta_{k}\}$ is a bounded sequence.
Let $\{\xi^{k}\}$ be the sequence defined in \eqref{def:xi}.
Define $\hat\varrho:=\|\xi^{1}\|+\big(\varrho+ \norm{\G^{-\frac{1}{2}}\N^{\frac{1}{2}}}\big)\E$, where $\varrho$ is defined by  \eqref{def:varrho}.
From \eqref{def:xi} and \eqref{ineq:xibound}, we know that for any $i\ge 1$,
\[
\label{ineq:Ny}
\|x^{i+1}_{e}\|_{\hat\Sigma_{f}+\S}^{2}
+\|y^{i+1}_{e}\|_{\N}^{2}
+{\hat\alpha\sigma}\|r^{i+1}\|^{2} \;\leq \;
\norm{\xi^{i+1}}^2 \;
\le\;
\hat\varrho^{2},
\]
where $\hat\alpha$ is defined in \eqref{def:beta}.
We then obtain that
\[
\label{ineq:SAx}
\begin{array}{l}
\|x_{e}^{i+1}\|_{\sigma\A\A^{*}}^{2}
\le
2\sigma\|r^{i+1}\|^{2}+2\|y_{e}^{i+1}\|_{\sigma\B\B^{*}}^{2}
\le
\frac{2}{\hat\alpha}(\hat\alpha\sigma)\|r^{i+1}\|^{2}+2\|y^{i+1}_{e}\|^{2}_{\N}
.
\end{array}
\]
By using \eqref{ineq:Ny} and \eqref{ineq:SAx} together, we get
\[
\label{ineq:xe-square}
\begin{array}{ll}
\|x_{e}^{i+1}\|^{2}_{\M}
&\le
\|x^{i+1}_{e}\|_{\hat\Sigma_{f}+\S}^{2}
+\frac{2}{\hat\alpha}(\hat\alpha\sigma)\|r^{i+1}\|^{2}
+2\|y_{e}^{i+1}\|_{\N}^{2}
\;\le\;
{2\hat\varrho^{2}}/{\hat \alpha}.
\end{array}
\]
We can see from \eqref{ineq:Ny} that $\|y^{i+1}_{e}\|_{\N}
\le\hat\varrho$. This, together with \eqref{ineq:xe-square} and the fact that $0<\hat\alpha<1$, implies
\[
\label{ineq:ineq1}
|\langle d_x^{i},x^{i+1}_{e}\rangle
+
\langle d_y^{i},y^{i+1}_{e}\rangle|
\le
(\sqrt{2/\hat\alpha}+1)\hat\varrho
\varepsilon_{k}.
\]
Note that
$\norm{\G^{-\frac{1}{2}} d_y^{k}}\le\norm{\G^{-\frac{1}{2}}\N^{\frac{1}{2}}} \varepsilon_{k}$ and $0<\alpha\le1$. Thus, we have
\[
\label{ineq:ineq2}
\alpha^{2}\|\G^{-\frac{1}{2}} (d_y^i-d_y^{i-1})\|^{2}\le
 2
 \norm{\G^{-\frac{1}{2}}\N^{\frac{1}{2}}}^{2}
(\varepsilon_{i}^2+\varepsilon_{i-1}^{2}).
\]
Therefore, by combining \eqref{ineq:ineq1} and \eqref{ineq:ineq2}, we can get
\[
\label{ineq:etak-etabar}
\begin{array}{lll}
\zeta_{k}
& \le &
\sum_{i=1}^{\infty}\big(
2|\langle d_x^{i},x^{i+1}_{e}\rangle
+
\langle d_y^{i},y^{i+1}_{e}\rangle|
+
\alpha^{2}\|d_{y}^{i}-d_{y}^{i-1}\|^{2}_{\G^{-1}}\big)
\\[2mm]
&\le&
\bar\zeta \; :=\; 2\big(\sqrt{2/\hat\alpha}+1\big)
\hat\varrho\E+4\norm{\G^{-\frac{1}{2}}\N^{\frac{1}{2}}}^{2}\E',
\end{array}
\]
where $\E$ and $\E'$ are defined in the beginning of Sec. \ref{convergence}.
By using \eqref{ineq:mainb}, \eqref{eq:deltay} and \eqref{ineq:etak-etabar}, we have that
\[
\label{3summable}
\begin{array}{rl}
& \sum_{k=1}^\infty \norm{\Delta_x^{k}}^{2}_{\F}
+\beta\sigma\norm{r^{k+1}}^{2}
+\norm{\Delta_y^{k}-\alpha\G^{-1}(d^k_y-d^{k-1}_y)}_{\G}^{2}
\\[2mm]
& \leq
\sum_{k=1}^\infty
(2|\inprod{ d_{x}^{k}}{x^{k+1}_e}|
+2|\inprod{ d_{y}^{k}}{y^{k+1}_e}|
+ \alpha^2\norm{d_y^k-d_y^{k-1}}_{\G^{-1}}^2)
\\[2mm]
&\quad+
\sum_{k=1}^\infty (\phi_{k}(\bar w)
-\phi_{k+1}(\bar w))
\le
\phi_{1}(\bar w) + \bar{\zeta},
\end{array}
\]
where $\bar\zeta$ is defined in \eqref{ineq:etak-etabar}.
Also, since $0<\alpha<1$, we have that
$$
\|\Delta_y^{k}+\alpha\G^{-1}(d_y^k- d_y^{k-1})\|_{\G}^{2}
\ge
\|\Delta_y^{k}\|_{\G}^{2}
-2 \|\Delta_y^k\|\|d_y^k- d_y^{k-1}\|.
$$
By \eqref{ineq:Ny} we know that $\|y^{i+1}_{e}\|_{\N}
\le \hat\varrho$.
Hence, $
\|\Delta_y^{k}\|
\le
2\|\N^{-\frac{1}{2}}\|\hat\varrho$.
From the fact that
$\norm{d_y^{k}}\le\norm{\N^{\frac{1}{2}}} \varepsilon_{k}$, we can get
$
\| (d_y^k-d_y^{k-1})\|\le \norm{\N^{\frac{1}{2}}} (\varepsilon_{k}+\varepsilon_{k-1}).
$
Thus from \eqref{3summable} and the above discussions, we have that
$$
\begin{array}{l}
\sum_{k=1}^{\infty}
(\|\Delta_x^{k}\|^{2}_{\M}
+\beta\sigma\|r^{k+1}\|^{2}
+\|\Delta_y^{k}\|_{\N}^{2})
\\[2mm]
\le
\max\big\{1,\|\M\F^{-1}\|,\|\N\G^{-1}\|\big\}
\sum_{k=1}^{\infty}
\big(\|\Delta_x^{k}\|^{2}_{\F}
+\beta\sigma\|r^{k+1}\|^{2}
+\|\Delta_y^{k}\|_{\G}^{2}\big)
\\[2mm]
\le
\omega_{0}:=
\max\big\{1,\|\M\F^{-1}\|,\|\N\G^{-1}\|\big\}\big(\phi_1(\bar{w})+\bar{\zeta}+4
\norm{\N^{\frac{1}{2}}}
\norm{\N^{-\frac{1}{2}}}\hat\varrho\E\big).
\end{array}
$$
Next, we estimate the value of $D(w^{k+1})$.
From  the fact that $\M \succeq\P_{x}^{k}$ and \eqref{opt:2sub},  we obtain that
$$
\begin{array}{l}
\dist^{2}(0,F(w^{k+1}))
\\[1mm]
\le\|d_{x}^{k}
+(\hat\Sigma_{f}+\S-\P_{x}^{k+1})\Delta_x^{k}
+
(\tau-1)\sigma\A r^{k+1}
-\sigma\A\B^{*}\Delta_y^{k}\|^{2}
\\[1mm]
\le\| d_{x}^{k}
+(\M-\P_{x}^{k+1})\Delta_x^{k}
+ \sigma\A (\tau r^{k+1}- r^k) \|^2
\\[1mm]
\leq 3\norm{\M} \big(  \norm{\M^{-\frac{1}{2}}d^k_x}^2 + \norm{\Delta_x^k}_\M^2
+ \sigma^2  \norm{\M^{-\frac{1}{2}}\A}^2 \norm{\tau r^{k+1}-r^k}^2
\big).
\end{array}
$$
Also, from the fact that $\N\succeq\P_{y}^{k}$  and \eqref{opt:2sub}, we have
$$
\begin{array}{l}
\dist^2(0,G(w^{k+1}))
\\[1mm]
\le
\|d_{y}^k
+(\hat\Sigma_g+\T-\P_{y}^{k+1})\Delta_y^k
+(\tau-1)\sigma\B r^{k+1}\|^{2}
\\[1mm]
\le\|d_{y}^k
+(\N-\P_{y}^{k+1})\Delta_y^k
+ \sigma\B  (\tau r^{k+1} - r^k +\A^*\Delta_x^k)\|^{2}
\\[1mm]
\le3\norm{\N} \big(  \norm{\N^{-\frac{1}{2}}d^k_y}^2 + \norm{\Delta_y^k}_\N^2+ 2\sigma^2  \norm{\N^{-\frac{1}{2}}\B}^2 \norm{\tau r^{k+1}-r^k}^2\big)
\\[1mm]
\quad +6\sigma^2 \norm{\B\A^*\M^{-\frac{1}{2}}}^2\norm{\Delta_x^{k}}_\M^2.
\end{array}
$$
Define
$
\omega_1:=\frac{(1+\tau^2)\sigma}{\beta}(\norm{\M^{-\frac{1}{2}}\A}^2 +
\norm{\N^{-\frac{1}{2}}\B}^2)
$
and
$$
\omega_2:=\max\big\{
\|\M\|+2\sigma^2 \norm{\B\A^*\M^{-\frac{1}{2}}}^2,
\|\N\|,
4\omega_1\big\}.
$$
It is now easy to verify from the above discussions that
$$
\begin{array}{ll}
\sum_{k=1}^{\infty} D(w^{k+1})
\le
\omega:=
3(\|\M\|+\|\N\|)\E'+2\omega_1\|r^1\|+3\omega_0\omega_2.
\end{array}
$$
Therefore, by the above inequality and Lemma \ref{lemma:sequence}, we know that \eqref{ineq:nec-r1} holds with $\omega>0$ being defined above. This completes the proof.
\qed
\end{proof}
\begin{remark}
(a)
We note that  the sequence $\{D(w^k)\}$ is not necessarily monotonically decreasing,
especially due to the inexact setting of the imsPADMM. Thus it is not surprising
that the iteration complexity result is established with the ``$\min_{1\leq i\leq k}$" operation
in Theorem \ref{non-ergodic-theorem}.
\\
(b)
For a majorized ADMM with coupled objective,
the non-ergodic complexity analysis was first proposed by Cui et al. \cite{cui}.
For the classic ADMM with separable objective functions,   Davis and Yin \cite{davis-yin} provided non-ergodic iteration complexity results in terms of the primal feasibility and the objective functions. One may refer to \cite[Remark 4.3]{cui} for a discussion on this topic.
\end{remark}

\section{Numerical Experiments}
\label{numerical}

In this section, we report the numerical performance of the sGS-imsPADMM for the following rather general convex QSDP (including SDP) problems
\begin{equation}
\label{eq-qsdp}
\begin{array}{ll}
\min \big\{ \frac{1}{2} \inprod{X}{\cQ X} + \inprod{C}{X}\,
\mid \,
\cA_E X   =  b_E,
\; \cA_I X    \geq b_I, \;
X \in \Sn_+\cap \cN  \big\}
\end{array}
\end{equation}
where  $\cS_+^n$ is the cone of $n\times n$ symmetric  positive semidefinite matrices in the space of $n\times n$ symmetric matrices $\cS^n$, $\cQ: \S^n\to\S^n$ is a self-adjoint positive semidefinite linear operator,  $\cA_E:\Sn \rightarrow \Re^{m_E}$ and $\cA_I:\Sn\rightarrow \Re^{m_I}$ are linear maps,  $C\in \Sn$,  $b_E\in \Re^{m_E}$ and $b_I\in \Re^{m_I}$ are given data,  $\cN$ is a nonempty simple closed convex set,
e.g., $\cN =\{X\in\Sn\;|\; L\leq X\leq U\}$ with $L,U\in \Sn$ being given matrices.
The dual of problem  (\ref{eq-qsdp}) is given by
\begin{eqnarray}
\label{eq-d-qsdp}
\begin{array}{rllll}
-\min &
\delta_{\cN}^*(-Z)
+\frac{1}{2}\inprod{W}{\cQ W}
-\inprod{b_E}{y_E}
-\inprod{b_I}{y_I} \\[1mm]
\mbox{s.t.}
&Z-\cQ W+S+\cA_E^* y_E+\cA_I^*y_I = C,
\
S\in \Sn_+, \  y_I\geq 0, \ W \in \W,
\end{array}
\end{eqnarray}
where
$\W$ is any subspace in $\S^{n}$ containing ${\rm Range}(\Q)$.
Typically $\W$ is chosen to be either $\S^n$ or ${\rm Range}(\Q)$.
Here we fix $\W=\S^n$.
In order to handle the equality and inequality constraints in \eqref{eq-d-qsdp} simultaneously, we add a slack variable $v$ to get the following equivalent problem:
\begin{equation}
\begin{array}{rllll}
\max &  \disp \big(-\delta_{\cN}^*(-Z) -\delta_{\Re_{+}^{m_I}}(v)\big) - \frac{1}{2}\inprod{W}{\cQ W} - \delta_{\Sn_{+}}(S) + \inprod{b_E}{y_E} + \inprod{b_I}{y_I}   \\[2mm]
 \mbox{s.t.} & Z  - \cQ W + S + \cA_E^* y_E + \cA_I^*y_I = C, \quad \D(v -  y_I) = 0, \quad W \in \W,
   \end{array}
   \label{eq-qsdp-dual}
\end{equation}
where $\cD\in \Re^{m_I\times m_I}$ is  a positive definite diagonal matrix
introduced for the purpose of scaling the variables.
The convex QSDP problem \eqref{eq-qsdp} is solved via its dual  \eqref{eq-qsdp-dual}
and we use $X\in\S^{n}$ and $u\in\Re^{m_{I}}$ to denote the
Lagrange multipliers
corresponding to the two groups of equality constraints in \eqref{eq-qsdp-dual}, respectively.
Note that if $\Q$ is vacuous, then \eqref{eq-qsdp} reduces to the
following general linear SDP:
\[
\min \big\{\inprod{C}{X}  \ |\
\cA_E X   =  b_E,
\  \cA_I X    \geq b_I, \
X \in \Sn_+\cap \cN
\}.
\label{eq-dnsdp}
\]
Its dual can be written as in \eqref{eq-qsdp-dual} as follows:
\begin{eqnarray}
\begin{array}{rllll}
\min &  \disp \big(\delta_{\cN}^*(-Z)+ \delta_{\Re_{+}^{m_I}}(v)\big) + \delta_{\Sn_{+}}(S) - \inprod{b_E}{y_E} - \inprod{b_I}{y_I}  \\[1mm]
 \mbox{s.t.} & Z + S + \cA_E^* y_E + \cA_I^*y_I = C, \
 \D(v-y_{I})=0,
   \end{array}
   \label{sdp-dual-relax}
\end{eqnarray}
or equivalently,
\begin{eqnarray}
\begin{array}{rllll}
\min & \disp \big(\delta_{\cN}^*(-Z)+ \delta_{\Re_{+}^{m_I}}(y_{I})\big) + \delta_{\Sn_{+}}(S) - \inprod{b_E}{y_E} - \inprod{b_I}{y_I}\\[1mm]
 \mbox{s.t.} & Z + S + \cA_E^* y_E + \cA_I^*y_I = C.
   \end{array}
   \label{sdp-dual}
\end{eqnarray}
Denote  the normal cone of $\N$ at $X$ by $N_{\N}(X)$. The KKT system of problem \eqref{eq-qsdp} is given by
\[
\label{eq:op-M}
\left\{
\begin{array}{l}
\A_{E}^{*}y_{E}+\A_{I}^{*}y_{I}+S+Z-\Q W-C=0, \  \A_{E}X-b_{E}=0,
\\[1mm]
0\in N_{\N}(X)+Z,\
\Q X-\Q W=0, \
X\in\S^{n}_{+},\
S\in\S_{+}^{n},\
\langle X,S\rangle=0,
\\[1mm]
\A_{I}X-b_{I}\ge0,\ y_{I}\ge 0,\  \langle A_{I}x-b,y_{I}\rangle =0.
\end{array}
\right.
\]
Based on the optimality condition \eqref{eq:op-M},
we measure the accuracy of a computed solution $(X,Z,W,S,y_E,y_{I})$ for QSDP \eqref{eq-qsdp} and its dual \eqref{eq-d-qsdp} via
\begin{equation}
\label{stop:sqsdp}
\begin{array}{c}
\eta_{\textup{qsdp}} =
\max\{\eta_D,\eta_{X},\eta_Z,  \eta_P,\eta_{W}, \eta_{S},
\eta_{I}\},
\end{array}
\end{equation}
where
$$
\begin{array}{l}
\eta_{D}=\frac{\|\A_{E}^{*}y_{E}+\A_{I}^{*}y_{I}+S+Z-\Q W-C\|}{1+\|C\|},\
\eta_{X}=\frac{\|X-\Pi_{\N}(X)\|}{1+\|X\|},\
\eta_{Z}=\frac{\|X-\Pi_{\N}(X-Z)\|}{1+\|X\|+\|Z\|},
\\[2mm]
\eta_{P}=\frac{\|\A_{E}X-b_{E}\|}{1+\|b_{E}\|},\
\eta_{W}=\frac{\|\Q X-\Q W\|}{1+\|\Q\|},
\eta_{S}=\max\big\{\frac{\|X-\Pi_{S_{+}^{n}}(X)\|}{1+\|X\|},
\frac{|\langle X,S\rangle|}{1+\|X\|+\|S\|}\big\},
\\[2mm]
\eta_{I}=\max\big\{
\frac{\|\min(0,y_{I})\|}{1+\|y_{I}\|},
\frac{\|\min(0,\A_{I}X-b_{I})\|}{1+\|b_{I}\|},
\frac{|\langle \A_{I}X-b_{I},y_{I}\rangle|}{1+\|\A_{I}x-b_{I}\|+\|y_{I}\|}\big\}.
\end{array}
$$
In addition, we also measure the objective value and the duality gap:
\begin{equation}
\label{gap}
\begin{array}{rl}
{\rm Obj_{primal}}&:=\frac{1}{2}\langle X,\Q X\rangle+\langle C,X\rangle,\\[1mm]
{\rm Obj_{dual}}&:=-\delta_{\cN}^*(-Z)-\frac{1}{2}\inprod{W}{\cQ W} + \inprod{b_E}{y_E} + \inprod{b_I}{y_I},
\\[1mm]
\eta_{\rm gap}&:=\frac{\rm Obj_{primal}-Obj_{dual}}{1+|\rm Obj_{primal}|+|Obj_{dual}|}.
\end{array}
\end{equation}
The accuracy of a computed solution $(X,Z,S,y_E,y_{I})$ for the SDP \eqref{eq-dnsdp}
is measured by a  relative residual $\eta_{\rm sdp}$ similar to the one defined
in \eqref{stop:sqsdp} but with $\Q$ vacuous.

Before we report our numerical results, we first
present some numerical techniques needed
for the efficient  implementations of our algorithm.

\subsection{On Solving Subproblems Involving Large Linear Systems of Equations}
In the course of applying
ADMM-type  methods to solve
\eqref{eq-qsdp-dual}, we often have to solve a large linear system of equations.
For example, for the subproblem corresponding to the block $y_I$,
the following subproblem with/without semi-proximal term  has to be solved:
\[
\label{yI-sub-problem}
\begin{array}{c}
\min\big\{
-\inprod{b_I}{y_I} + \frac{\sigma}{2}\norm{ [\cA_I,-\D]^* y_I - r}^2 +
\frac{1}{2}\|y_I-y_I^{-}\|^{2}_{\T} \big\},
\end{array}
\]
where $\T$ is a self-adjoint positive semidefinite linear operator on $\Re^{m_I}$,  and
$r$ and $y_I^{-}$  are given data.
Note that solving \eqref{yI-sub-problem} is equivalent to  solving
\[
\label{q-equality}
  \big(\sigma (\cA_I\cA_I^*+\D^2)  + \T\big) y_I = \tilde{r} := b_I + \sigma(\cA_I,-\D)r + \T y_I^{-}.
\]
It is generally very  difficult to
compute the solution of \eqref{q-equality} exactly
for large scale problems if $\T$ is the zero operator, i.e., not adding a proximal term.
We now provide an approach
for choosing the proximal term $\T$, which is based on a
technique\footnote{ This technique is originally designed for choosing the preconditioner for the preconditioned conjugate gradient method when solving large-scale linear systems, but it can be directly applied to choosing the semi-proximal terms when solving subproblems in ALM-type or ADMM-type methods. One may refer to \cite[Section 4.1]{abcd} for the details.}
 proposed by Sun et al. \cite{abcd}, so that the problem can be efficiently
solved exactly  while ensuring that $\T$ is not too ``large".
Let $$\V:=\cA_I\cA_I^*+\D^2.$$
Suppose that $\V$ admits the  eigenvalue decomposition
$\V=\mbox{$\sum_{i=1}^{n}$} \lambda_{i}\P_{i}\P_{i}^{*}$,  with
$\lambda_{1}\ge\ldots\ge \lambda_{n}\ge 0$.
We can choose $\T$ by using the first $l$ largest  eigenvalues and the corresponding eigenvectors of $\V$.
By following the procedure provided in \cite{abcd} we have
\[
\label{T1}
\begin{array}{l}
\T= \sigma \mbox{$\sum_{i=l+1}^{n}$}
(\lambda_{l+1}-\lambda_{i})\P_{i}\P_{i}^{*}.
\end{array}
\]
Thus $\T$ is self-adjoint positive semidefinite.
Moreover, it is more likely that such a $\T$ is ``smaller'' than
the natural choice of
setting it to be $\sigma(\lambda_1 \I -\V)$.
Indeed we have observed in our numerical experiments that the latter choice
always leads to more iterations compared to the choice in \eqref{T1}.

To solve \eqref{q-equality}, we need to compute $(\sigma\V+\T)^{-1}$,
which can be obtained analytically as
$
(\sigma\V+\T)^{-1}
=(\sigma\lambda_{l+1})^{-1}\I+\sum_{i=1}^{l}(
(\sigma\lambda_{i})^{-1}-(\sigma\lambda_{l+1})^{-1})\P_i\P_i^{*}
$.
Thus, we only need to calculate the first few largest eigenvalues and the corresponding eigenvectors of $\V$ and this can be done  efficiently via variants of the Lanczos method.
Finally, we add that
when the problem \eqref{yI-sub-problem} is allowed to be solved inexactly, we can
set $\T=0$ in \eqref{yI-sub-problem}  and  solve the linear system
$\sigma \V =\tilde{r}$
by a preconditioned conjugate gradient (PCG) method. In this setting, $(\sigma\V+\T)^{-1}$ with $\T$ defined in \eqref{T1} can serve as an effective  preconditioner.

Note that we can apply similar techniques to solve large linear systems of equations arising from solving the subproblems corresponding to $W$ in
\eqref{eq-qsdp-dual}.

\subsection{
Numerical Results for Quadratic/Linear SDP Problems}

In our numerical experiments, we construct QSDP test instances
based on the doubly nonnegative SDP problems arising from relaxation of binary integer quadratic (BIQ) programming with a large number of inequality constraints that was introduced by Sun et al. \cite{sun13} for getting tighter bounds.
The problems that we actually solve have the following form:
$$
\begin{array}{rl}
\min & \frac{1}{2}\inprod{X}{\cQ X} + \frac{1}{2} \inprod{Q}{\overline{X}} + \inprod{c}{x}  \\[5pt]
   \mbox{s.t.}
       & \textup{diag}(\overline{X}) - x = 0, \
       X = \begin{pmatrix}
\overline{X} & x \\
x^T & 1 \\
\end{pmatrix}
\in \Sn_+, \
X\in \N:= \{X\in \Sn:\; X\ge 0\},\\
&x_i-\overline{X}_{ij}\geq 0,\,
 x_j-\overline{X}_{ij} \geq 0,\,
 \overline{X}_{ij} - x_i -x_j \geq -1,\,
\forall\, 1\leq i < j \leq n-1.
\end{array}
$$
For convenience, we call them {as} QSDP-BIQ problems. When $\Q$ is vacuous, we call the corresponding linear SDP problems as SDP-BIQ problems.
The test data for {$Q$ and $c$ are taken from the} Biq Mac Library maintained by Wiegele\footnote{\url{http://biqmac.uni-klu.ac.at/biqmaclib.html}}.

We tested one group of SDP-BIQ problems and three groups of QSDP-BIQ problems with each group consisting of $80$ instances with $n$ ranging from $151$ to $501$.
We  compare the performance of the sGS-imsPADMM to the directly extended multi-block sPADMM {with the aggressive step-length of 1.618}
on solving these SDP/QSDP-BIQ problems.
 We should mention that, although its convergence is not guaranteed, such a directly extended sPADMM is currently more or less the benchmark among first-order methods
targeting to solve multi-block linear and quadratic SDPs to modest accuracy. Note that for QSDP/SDP problems, the majorization step is not necessary, so we shall henceforth call the sGS-imsPADMM as sGS-isPADMM.
We have implemented our sGS-isPADMM  and the directly extended sPADMM in {\sc Matlab}.
All the $320$ problems are tested on a HP Elitedesk
with one Intel Core i7-4770S Processor (4 Cores, 8 Threads, 8M Cache, 3.1 to 3.9 GHz) and 8 GB RAM.
We solve the QSDP \eqref{eq-qsdp} and the SDP \eqref{eq-dnsdp} via their
duals \eqref{eq-qsdp-dual} and \eqref{sdp-dual-relax}, respectively,
 where we set
$\D:=\alpha\I$ with $\alpha=\sqrt{\|\A_{I}\|}/2$.
We adopt a similar strategy used in \cite{sun13,li14} to adjust the step-length $\tau$ \footnote{If {$\tau\in[\frac{1+\sqrt{5}}{2},\infty)$} but it holds that $\sum_{k=0}^{\infty}\|\Delta_y^{k}\|^{2}+\|r^{k+1}\|^{2}<\infty$, the imsPADMM (or the sGS-imsPADMM) is also convergent. By making minor modifications to the proofs in Sec.  \ref{convergence} and using the fact that $\sum_{k=0}^{\infty}\|d_{y}^{k}\|<\infty$, we can get this convergence result with ease. We {omit} the detailed proof to reduce the length of this paper and one may refer to \cite[Theorem 1]{admmnote} for such a result and its detailed proof for the sPADMM setting. During our numerical implementation we always use $\tau\in[1.618,1.95]$.}.
The sequence $\{\varepsilon_{k}\}_{k\ge 0}$  that we used in imspADMM is chosen such that $\varepsilon_{k}\le 1/k^{1.2}$. The maximum iteration number is set at $200,000$.

We compare sGS-isPADMM applied to \eqref{eq-qsdp-dual} with a 5-block directly extended sPADMM (called sPADMM5d) with step-length of $1.618$ applied on \eqref{eq-d-qsdp}, and compare sGS-isPADMM applied to \eqref{sdp-dual-relax} with a 4-block directly extended sPADMM (called sPADMM4d) provided by  \cite{sun13} with step-length\footnote{We did not test the sPADMM4d with $\tau=1$ as it has been verified in \cite{sun13} that it almost always takes $20\%$ to $50\%$ more time than the one with $\tau=1.618$.} of 1.618 on \eqref{sdp-dual}.
For the comparison between sPADMM4d and some other  ADMM-type methods
in solving linear SDP problems, one may refer to \cite{sun13} for the details.
For the
sGS-isPADMM applied to \eqref{eq-qsdp-dual}, the subproblems corresponding to the blocks $(Z,v)$ and $S$ can be
solved analytically by computing the projections onto $\N\times \Re^{m_I}_+$ and $\S^n_+$, respectively.
For the subproblems corresponding to {$y_E$}, we solve the
linear system of equations involving the coefficient matrix
$\cA_E \cA_E^*$ via its Cholesky factorization since this computation can be done
without incurring excessive cost and memory.
For
the subproblems corresponding to $y_{I}$ and $W$,
we need to solve very large scale linear systems of equations and they are
solved via a preconditioned conjugate gradient (PCG) method
with preconditioners that are described in the previous subsection.
In the implementation of the sGS-isPADMM, we have used
the strategy described in Remark \ref{remark4} (b)
to decide whether the quadratic subproblems in each of the forward GS sweeps
should be solved. In our numerical experiments, we have found that
very often, the quadratic subproblems in the forward sweep actually
need not be solved as the  solutions computed in the prior
backward sweep already are good approximate solutions
to those subproblems. Such a strategy, which is the consequence of the flexibility
allowed by the inexact minimization criteria in sGS-isPADMM, can offer significant computational
savings especially when the subproblems have to be solved by {a} Krylov iterative solver
such as the PCG method.  We note that in the event when a quadratic subproblem
in the forward or backward sweep has to be solved by a PCG method, the solution computed in the
prior sweep or cycle should be {used} to serve as a good initial starting point for the
PCG method.

For {the} sPADMM5d applied to {\eqref{eq-d-qsdp}, the subproblems involving the blocks $Z$, $S$ and {$y_E$} can be solved
just as in the case of the  sGS-isPADMM. For the subproblems corresponding to the nonsmooth block $y_I$, since these subproblems must be solved exactly, a proximal term whose Hessian is $\lambda_{\max}\I-\sigma\A_{I}\A_{I}^{*}$ (with $\lambda_{\max}$ being the largest eigenvalue of $\sigma\A_{I}\A_{I}^{*}$) has to be added to ensure that an exact solution can be computed efficiently.
Besides, we can also apply a directly extended 5-block sPADMM (we call it sPADMM5d-2 for convenience) on \eqref{eq-qsdp-dual}. In this case, we can use the proximal term described in \eqref{T1} in the previous subsection, where $l$ is typically chosen to be less than $10$.
We always choose a similar
proximal term when solving the subproblems corresponding to $W$ for both the sPADMM5d and the sPADMM5d-2. Since the performance of the sPADMM5d-2 applied to \eqref{eq-qsdp-dual} is very similar to that of the sPADMM5d  applied to \eqref{eq-d-qsdp}, we only report our numerical results for the latter.

{\begin{figure}
\includegraphics[width=0.45\textwidth]{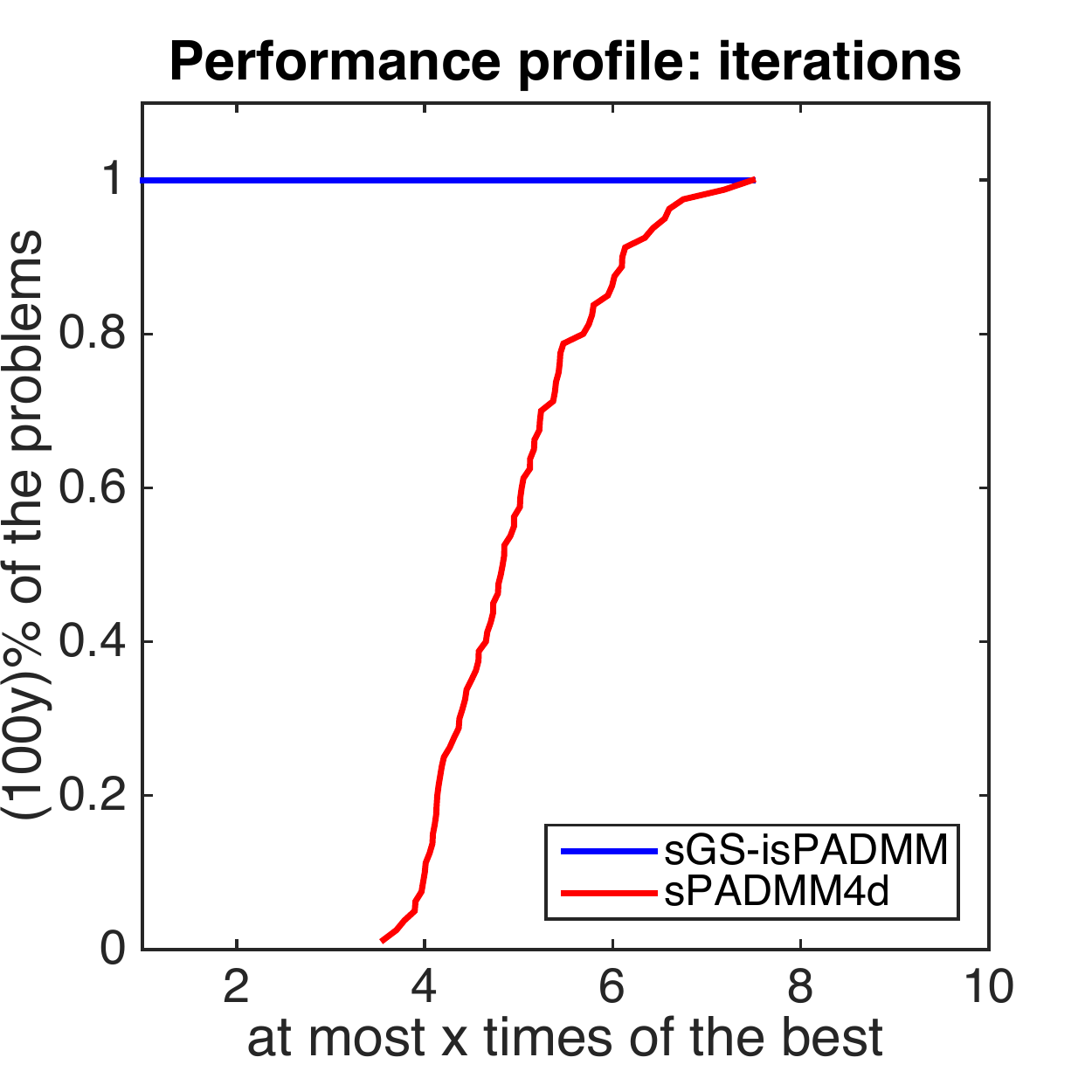}
\quad
\quad
\includegraphics[width=0.45\textwidth]{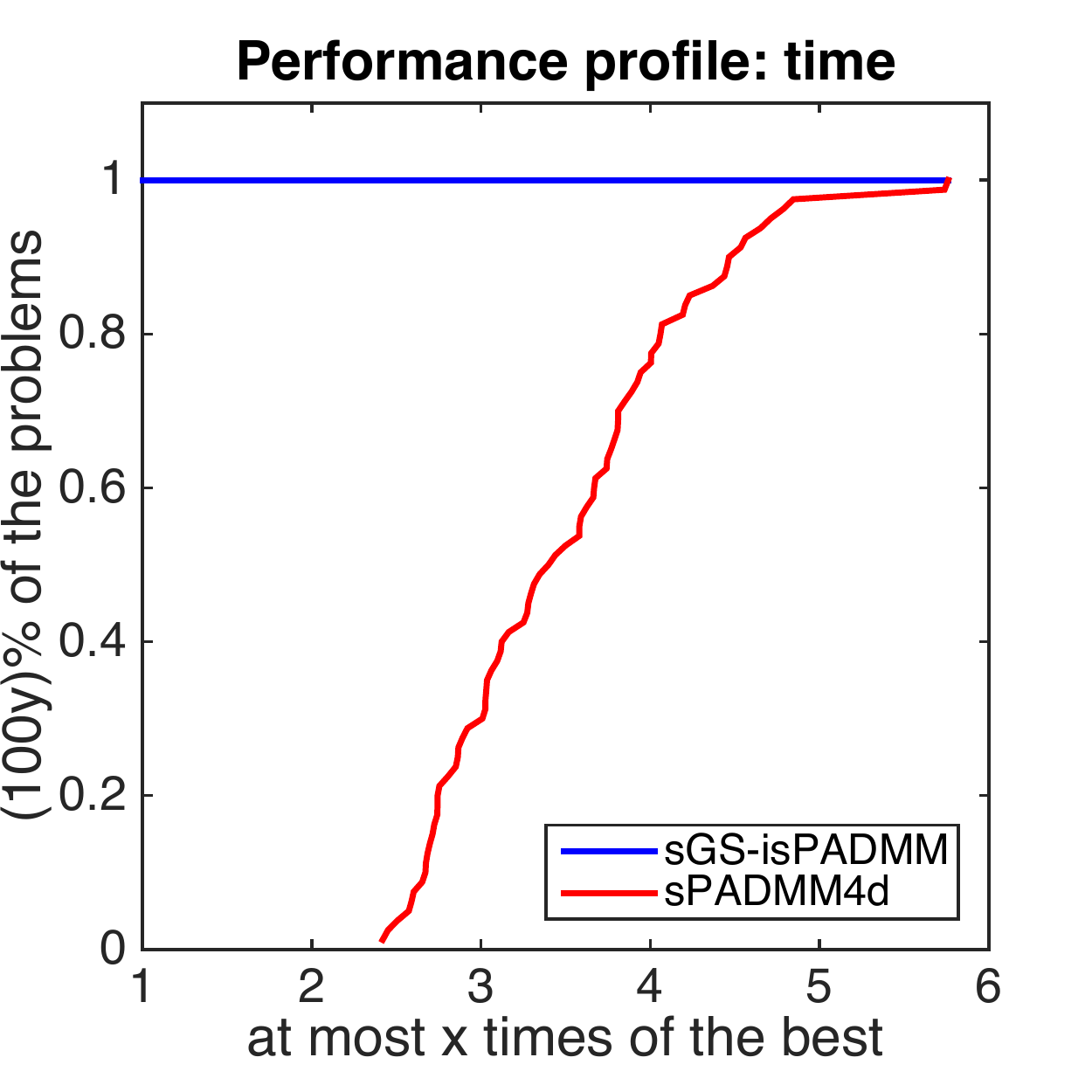}
\caption{Performance profiles of sGS-isPADMM and sPADMM4d
on solving SDP-BIQ problems. \label{figure-sdp}}
\

\includegraphics[width=0.45\textwidth]{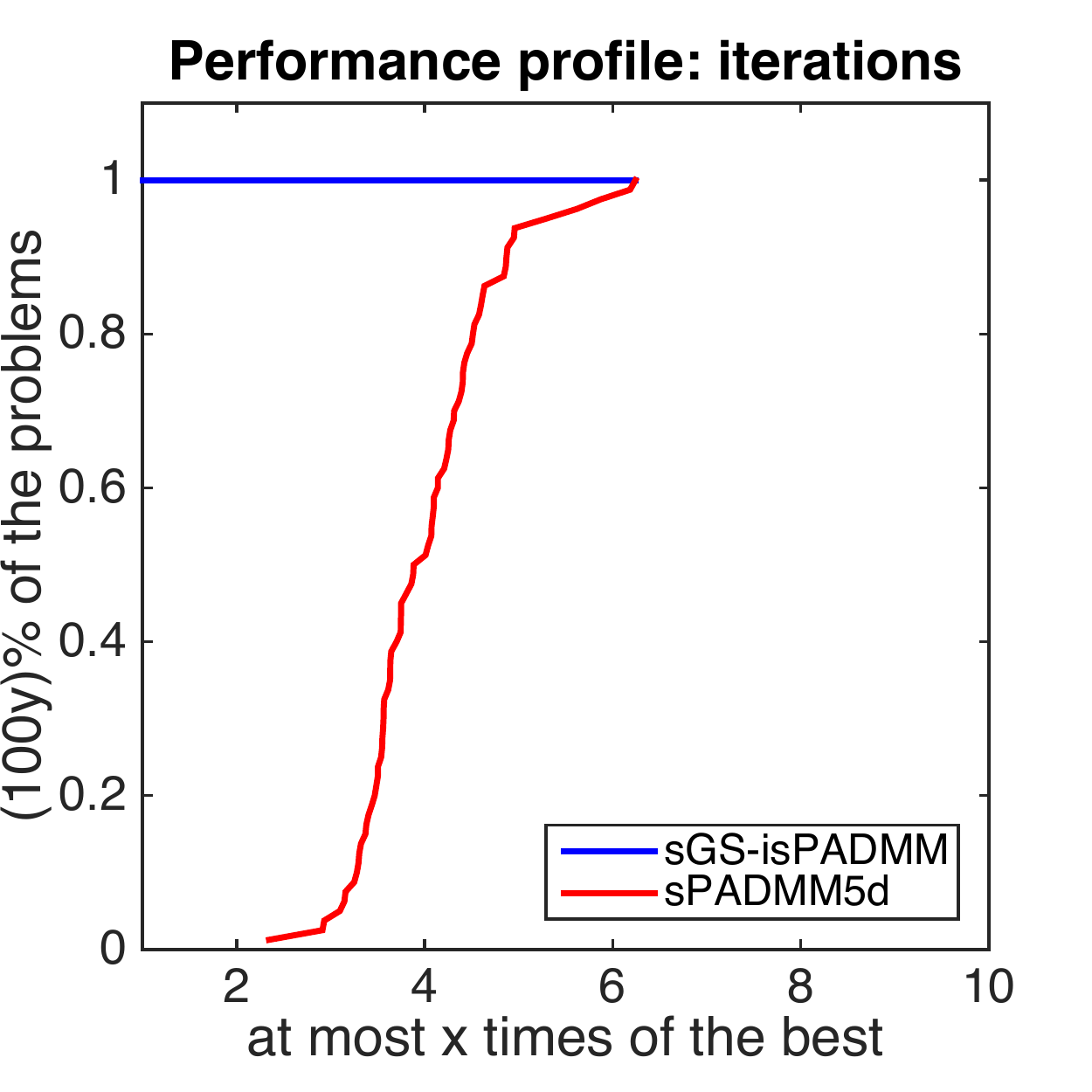}
\quad
\quad
\includegraphics[width=0.45\textwidth]{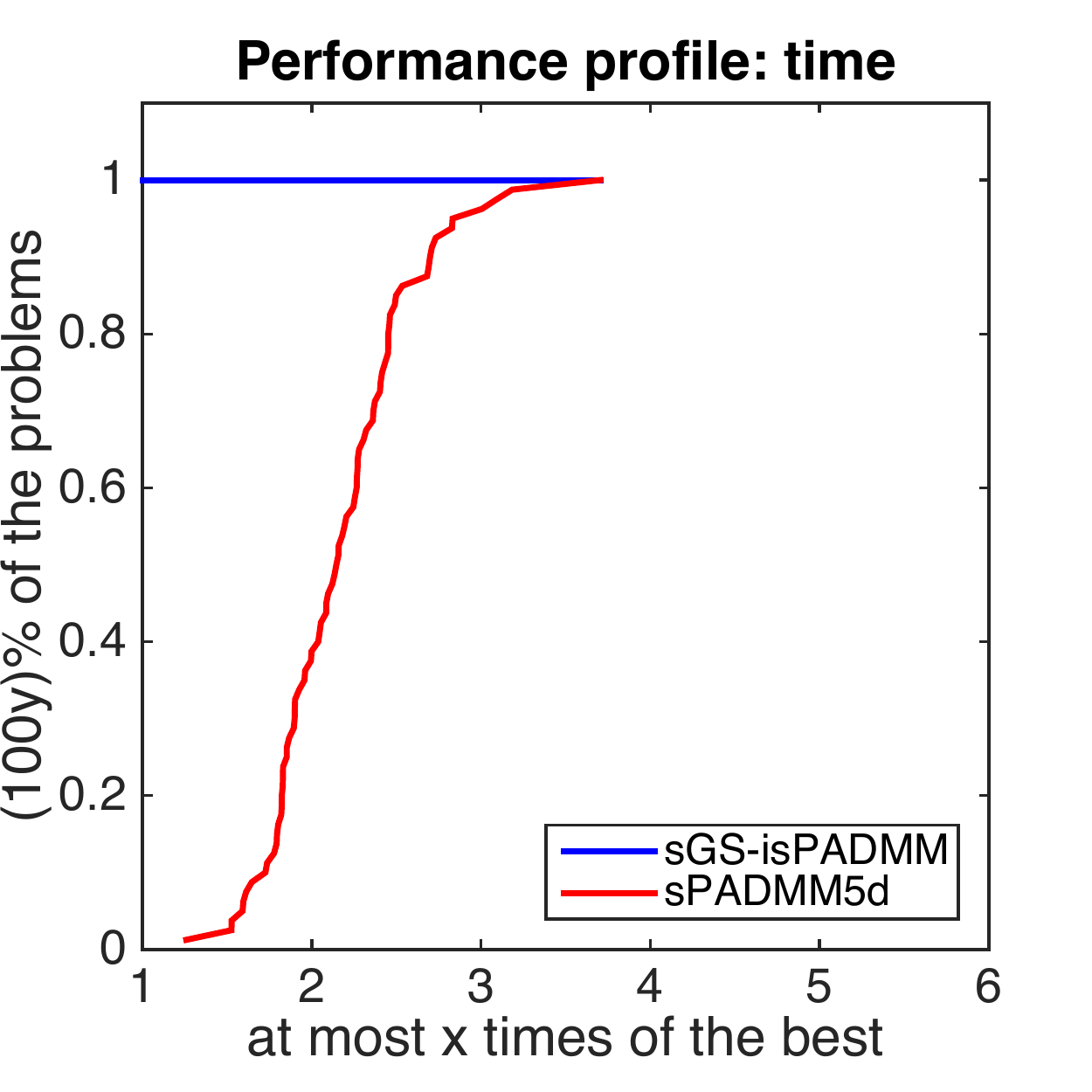}
\caption{Performance profiles of sGS-isPADMM and sPADMM5d
on solving QSDP-BIQ problems (group 1).\label{figure-qsdp1}}
\end{figure}

\begin{figure}
\includegraphics[width=0.45\textwidth]{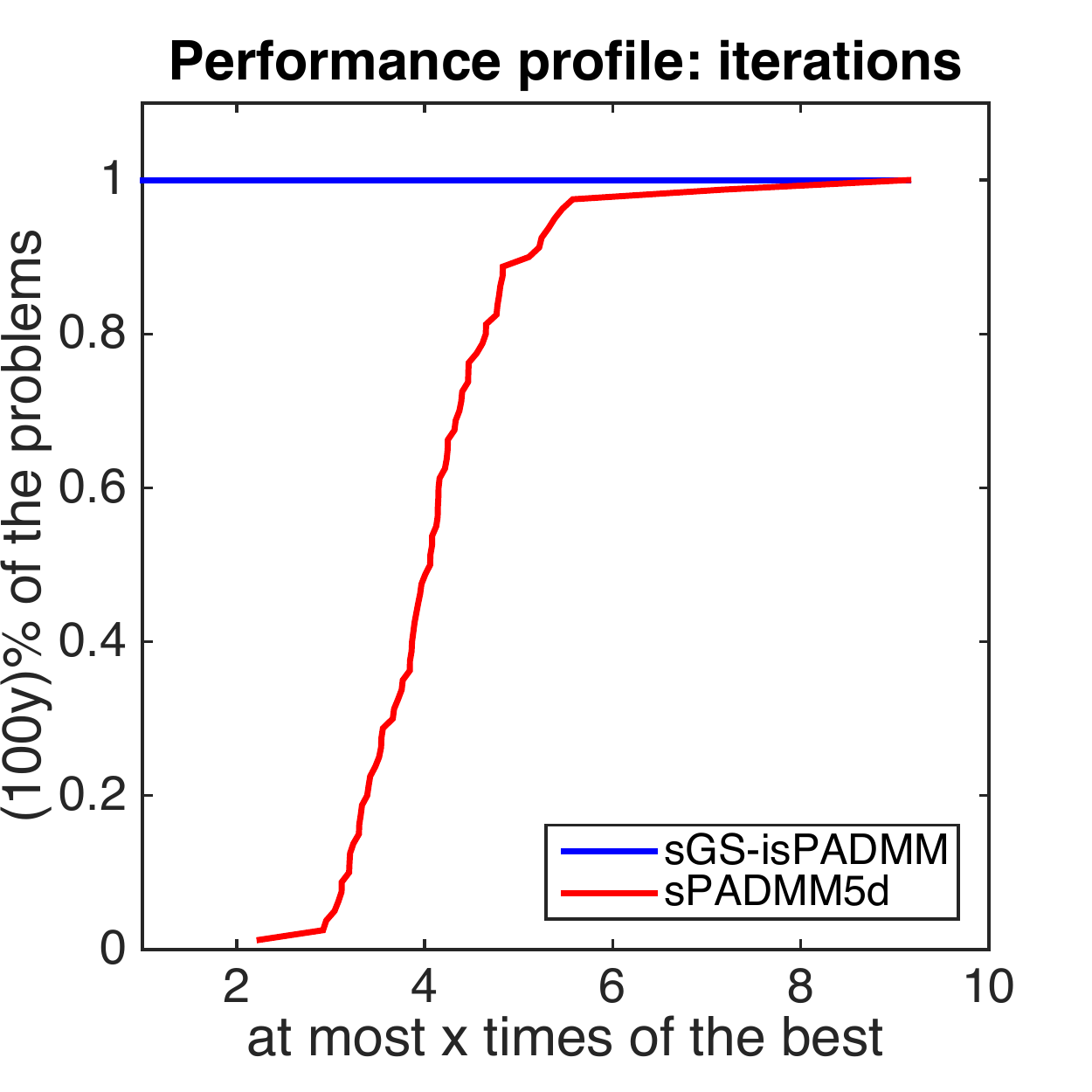}
\quad
\quad
\includegraphics[width=0.45\textwidth]{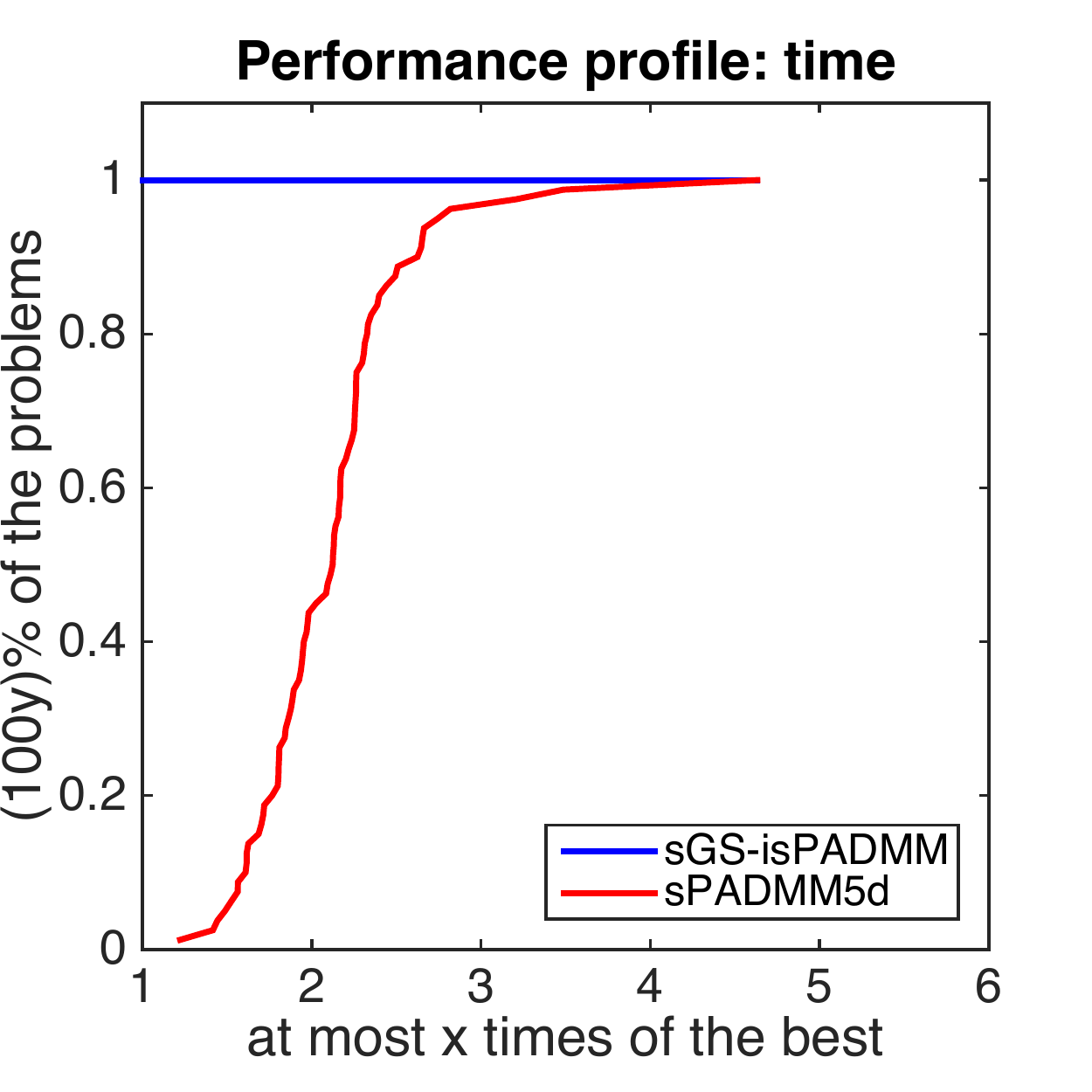}
\caption{\scriptsize{Performance profiles of sGS-isPADMM and sPADMM5d
on solving QSDP-BIQ problems (group 2).\label{figure-qsdp2}}}
\

\

\includegraphics[width=0.45\textwidth]{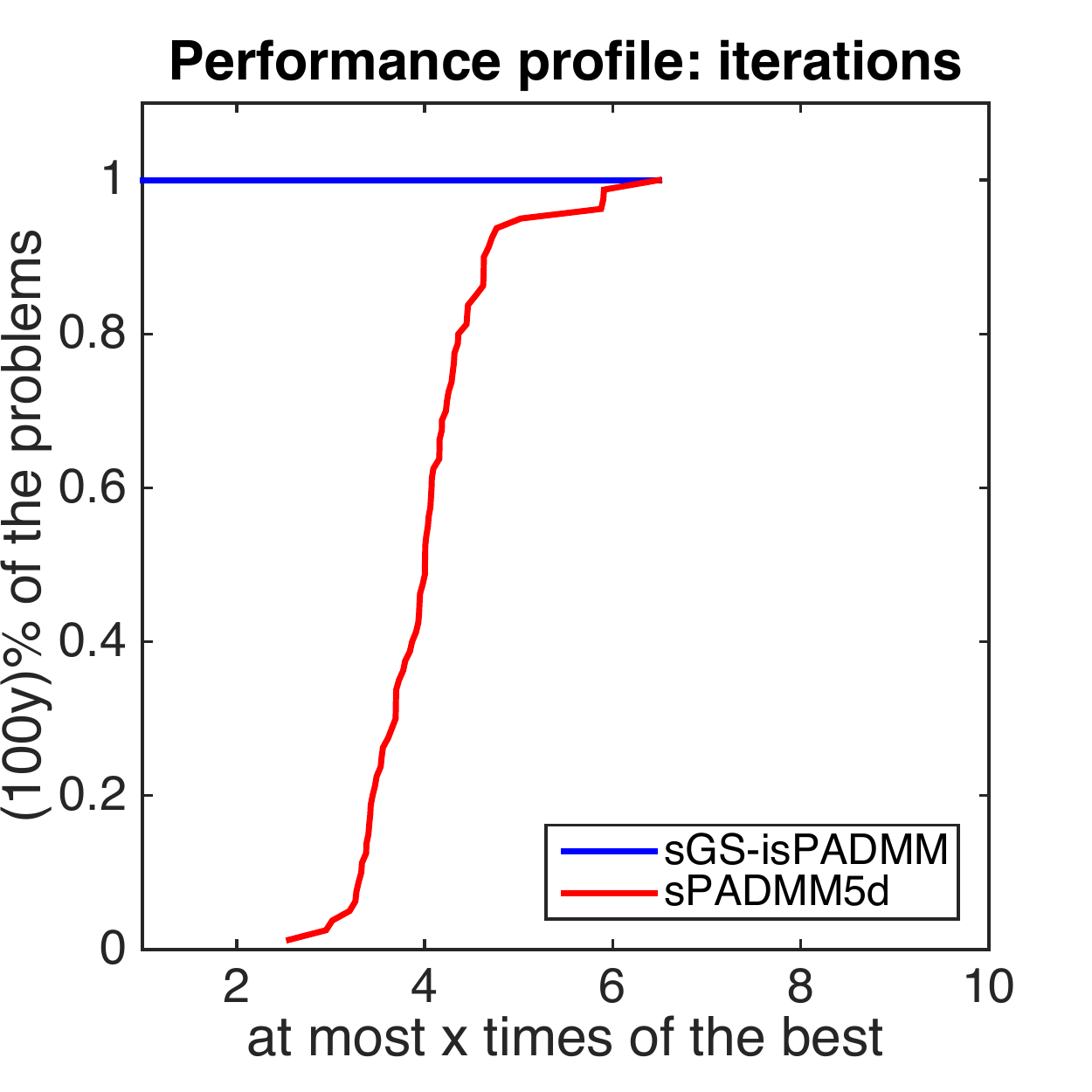}
\quad
\quad
\includegraphics[width=0.45\textwidth]{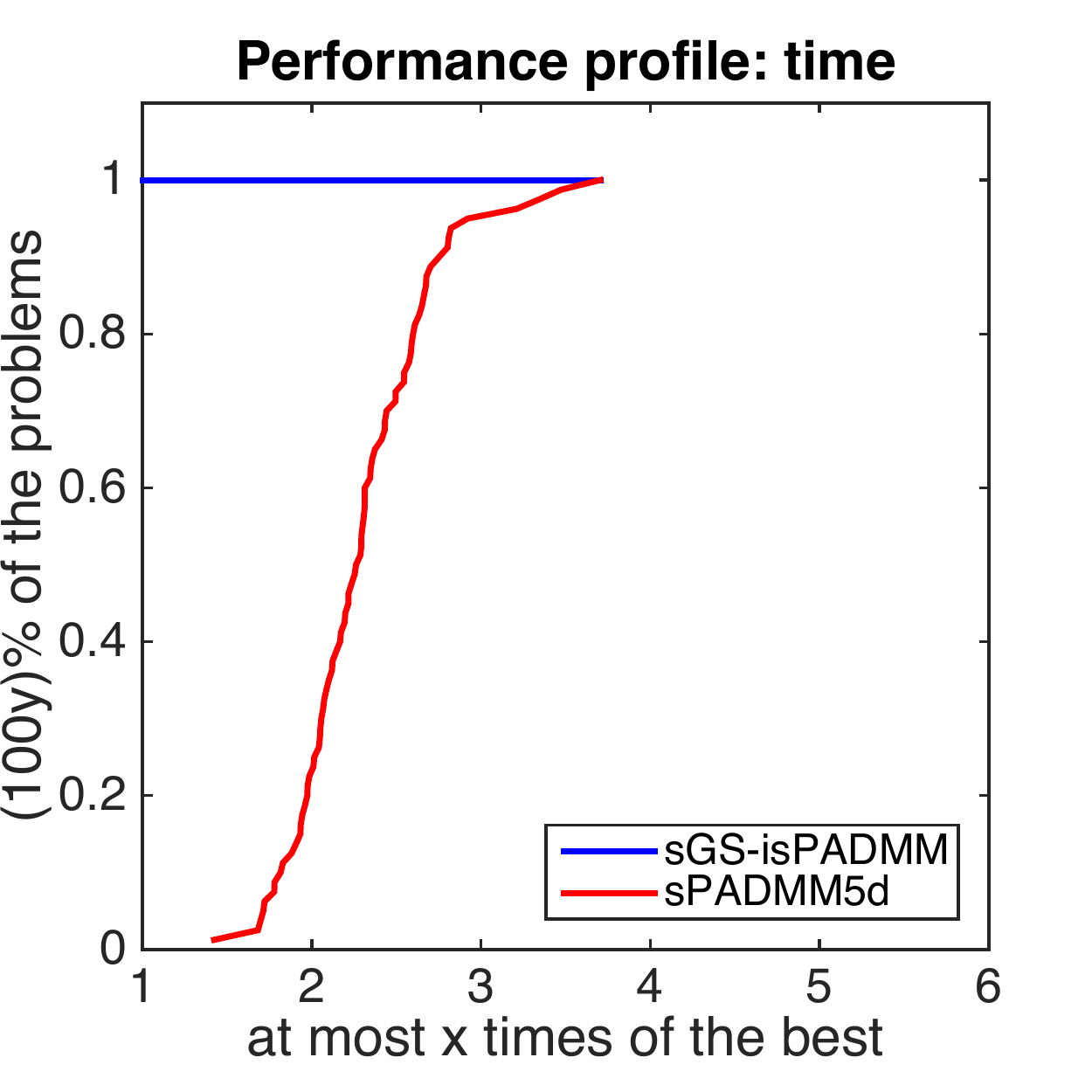}
\caption{\scriptsize{Performance profiles of sGS-isPADMM and sPADMM5d
on solving QSDP-BIQ problems (group 3).\label{figure-qsdp3}}}
\end{figure}

We now report our numerical results.
Figure \ref{figure-sdp} shows the numerical performance of the sGS-isPADMM and sPADMM4d
in solving SDP-BIQ problems to the accuracy of $10^{-6}$ in $\eta_{qsdp}$.
One can observe that sGS-isPADMM is $3$ to $5$ times faster than the sPADMM4d,  on approximately $80\%$ problems in terms of computational time.
Figure \ref{figure-qsdp1} shows the numerical performance of  the sGS-isPADMM and sPADMM5d in solving QSDP-BIQ problems (group 1) to the accuracy of $10^{-6}$ in $\eta_{qsdp}$.
For this group of tested instances, $\Q$ is chosen as the symmetrized Kronecker operator $\Q(X)=\frac{1}{2}(AXB+BXA)$, as was used in \cite{toh08}, with $A$, $B$ being randomly generated symmetric positive semidefinite matrices such that ${\rm rank}(A)=10$ and ${\rm rank}(B)\approx n/10$. Clearly $\Q$ is self-adjoint and positive semidefinite on $\S^{n}$ but highly rank deficient \cite[Appendix]{TTT98}.
As is showed in Figure \ref{figure-qsdp1}, the sGS-isPADMM is at least $2$ times faster than the sPADMM5d, on solving the vast majority the tested problems in terms of computational time.
Figure \ref{figure-qsdp2} demonstrates the numerical performance of the sGS-isPADMM and sPADMM5d in solving QSDP-BIQ problems (group 2) to the accuracy of $10^{-6}$ in $\eta_{qsdp}$.
Here, $\Q$ is chosen as the symmetrized Kronecker operator $\Q(X)=\frac{1}{2}(AXB+BXA)$ with $A$, $B$ being matrices truncated from two different large correlation matrices ({\tt Russell 1000 {\rm and} Russell 2000}) fetched from Yahoo Finance by {\sc Matlab}.
As can be seen from Figure \ref{figure-qsdp2},  sGS-isPADMM is $2$ to $3$ times faster than the sPADMM5d  for about $80\%$ of the problems in terms of computational time.
Figure \ref{figure-qsdp3} shows the numerical performance of the sGS-isPADMM and sPADMM5d
in solving QSDP-BIQ problems (group 3) to the accuracy of $10^{-6}$ in $\eta_{qsdp}$.
Here, $\Q$ is chosen as the Lyapunov operator $\Q(X)=\frac{1}{2}(AX+XA)$ with $A$ being a randomly generated symmetric positive semidefinite matrix such that ${\rm rank}(A)\approx n/10$.
We should note that for these instances, the quadratic subproblems
corresponding to $W$ in both the sGS-isPADMM and ADMM5d can admit closed-form solutions by using the eigenvalue decomposition of $A$
and adapting the technique in \cite[Section 5]{Todd99} to compute
them.
In our numerical test, we compute the closed-form solutions for these subproblems.
As can be seen from Figure \ref{figure-qsdp3},  the sGS-isPADMM is $2$ to $4$ times faster than the sPADMM5d  for more than $90\%$ tested instances  in terms of computational time.
Table \ref{table:qsdp} gives the detailed computational results for selected tested instances with $n=501$.
The full table for all the 320 problems is available at the authors' website\footnote{
\url{http://www.math.nus.edu.sg/~mattohkc/publist.html}}.

To summarize, we have seen that our sGS-isPADMM is typically 2 to 3 times faster
than the directly extended multi-block sPADMM even with the aggressive step-length of
$1.618$. We achieve this by exploiting the flexibility allowed by our
proposed method for only requiring approximate solutions to the subproblems
in each iteration. In contrast, the directly extended sPADMM must
solve the subproblems exactly, and hence is forced to add
appropriate proximal terms which may slow down the convergence. Indeed, we observed that its convergence is often much slower than that of
the sGS-isPADMM.
The merit that is brought about by solving the original subproblems inexactly without adding proximal terms  is thus evidently clear.

\begin{table}
\begin{scriptsize}
\begin{center}
\caption{
\label{table:qsdp}
The numerical performance of sGS-isPADMM and the directly extended multi-block ADMM with step-length $\tau=1.618$ on $1$ group of SDP-BIQ problems and $3$ groups of QSDP-BIQ problems with $n>500$ (accuracy $=10^{-6}$)}
\begin{tabular}{l c c l l l l}

\hline
\mc{1}{l}{Problem}
&\mc{1}{l}{$m_E;m_{I}$}
&\mc{1}{l}{$n_s$}
&\mc{1}{l}{Iteration}
&\mc{1}{l}{$\eta_{qsdp}$}
&\mc{1}{l}{$\eta_{gap}$}
&\mc{1}{l}{Time}
\\[2pt]
\mc{1}{l}{}
&\mc{2}{l}{}
&\mc{1}{l}{sGS-isP$|$sP-d}
&\mc{1}{l}{sGS-isP$|$sP-d}
&\mc{1}{l}{sGS-isP$|$sP-d}
&\mc{1}{l}{sGS-isP$|$sP-d}
\\
\hline

\hline
\mc{7}{l}{SDP-BIQ}\\[1pt]
bqp500-2 &501$;$374250&501 &17525$|$82401
&9.9-7$|$9.9-7
&-6.3-7$|$2.3-8
&42:27$|$2:12:29
\\[1pt]   
bqp500-4 &501$;$374250&501 &15352$|$75995
&9.9-7$|$9.9-7
&-6.4-7$|$-3.2-8
&36:53$|$1:59:52
\\[1pt]   
bqp500-6 &501$;$374250&501 &17747$|$78119
&9.9-7$|$9.9-7
&-1.6-7$|$-2.4-8
&45:10$|$2:04:23
\\[1pt]   
bqp500-8 &501$;$374250&501 &20386$|$110825
&9.9-7$|$9.9-7
&-4.3-7$|$2.1-8
&52:04$|$3:10:43
\\[1pt]   
bqp500-10 &501$;$374250&501 &16407$|$68985
&9.7-7$|$9.9-7
&-5.6-7$|$3.7-9
&39:30$|$1:46:01
\\[1pt]   
gka1f &501$;$374250&501 &9101$|$60073
&9.9-7$|$9.9-7
&-4.4-7$|$1.1-8
&20:22$|$1:32:22
\\[1pt]   
gka2f &501$;$374250&501 &16193$|$74034
&9.9-7$|$9.9-7
&-2.7-7$|$-1.1-8
&39:35$|$1:59:59
\\[1pt]   
gka3f &501$;$374250&501 &16323$|$72563
&9.9-7$|$9.9-7
&-1.3-7$|$3.9-8
&40:38$|$1:56:28
\\[1pt]   
gka4f &501$;$374250&501 &15502$|$63285
&9.6-7$|$9.9-7
&-6.1-7$|$3.4-8
&36:58$|$1:41:20
\\[1pt]   
gka5f &501$;$374250&501 &17664$|$76164
&9.9-7$|$9.9-7
&-1.3-7$|$1.1-8
&43:45$|$2:05:14
\\[1pt]

\hline

\mc{7}{l}{QSDP-BIQ (group 1)} \\[1pt]
bqp500-2 &501$;$374250&501 &19053$|$71380
&9.9-7$|$9.9-7
&-1.2-7$|$1.1-8
& 1:02:31$|$1:52:02
\\[1pt]    
bqp500-4 &501$;$374250&501 &13905$|$67865
&9.9-7$|$9.9-7
&-8.9-7$|$7.8-8
&43:17$|$1:46:07
\\[1pt]    
bqp500-6 &501$;$374250&501 &17211$|$62562
&9.9-7$|$9.9-7
&-2.0-7$|$6.9-8
&56:23$|$1:37:19
\\[1pt]    
bqp500-8 &501$;$374250&501 &19742$|$85057
&9.9-7$|$9.9-7
&-4.9-7$|$7.0-8
& 1:05:09$|$2:15:52
\\[1pt]    
bqp500-10 &501$;$374250&501 &17690$|$65484
&9.9-7$|$9.9-7
&-2.3-7$|$6.7-8
&58:00$|$1:43:04
\\[1pt]    
gka1f &501$;$374250&501 &8919$|$55669
&9.9-7$|$9.9-7
&-8.8-7$|$4.1-8
&26:42$|$1:25:01
\\[1pt]    
gka2f &501$;$374250&501 &13587$|$61324
&9.9-7$|$9.9-7
&-4.5-7$|$2.1-8
&42:50$|$1:37:15
\\[1pt]    
gka3f &501$;$374250&501 &13786$|$62438
&9.9-7$|$9.9-7
&-2.2-7$|$3.1-8
&42:55$|$1:37:29
\\[1pt]    
gka4f &501$;$374250&501 &13953$|$57164
&9.6-7$|$9.9-7
&-7.2-7$|$-3.4-8
&44:25$|$1:31:14
\\[1pt]    
gka5f &501$;$374250&501 &15968$|$62001
&9.9-7$|$9.9-7
&-1.4-7$|$4.6-8
&50:22$|$1:35:40
\\[1pt]

\hline

\mc{7}{l}{QSDP-BIQ (group 2)}\\[1pt]
bqp500-2 &501$;$374250&501 &16506$|$79086
&9.9-7$|$9.9-7
&-1.2-7$|$4.2-8
&52:46$|$1:52:08
\\[1pt]   
bqp500-4 &501$;$374250&501 &8675$|$30677
&9.9-7$|$9.9-7
&2.7-8$|$2.3-8
&25:32$|$41:15
\\[1pt]   
bqp500-6 &501$;$374250&501 &10043$|$42654
&9.9-7$|$9.9-7
&-3.0-8$|$8.3-8
&29:46$|$58:58
\\[1pt]   
bqp500-8 &501$;$374250&501 &9410$|$43785
&9.9-7$|$9.9-7
&-2.5-8$|$2.9-8
&27:37$|$59:05
\\[1pt]   
bqp500-10 &501$;$374250&501 &10656$|$35213
&9.9-7$|$9.9-7
&-3.6-8$|$8.8-8
&32:35$|$47:00
\\[1pt]   
gka1f &501$;$374250&501 &10939$|$52226
&9.9-7$|$9.9-7
&-5.8-8$|$3.8-8
&36:10$|$1:16:48
\\[1pt]   
gka2f &501$;$374250&501 &7757$|$34660
&9.9-7$|$9.9-7
&-1.8-8$|$6.0-8
&25:17$|$48:40
\\[1pt]   
gka3f &501$;$374250&501 &11241$|$45857
&9.9-7$|$9.9-7
&-1.2-8$|$2.7-8
&34:55$|$1:02:59
\\[1pt]   
gka4f &501$;$374250&501 &11706$|$37466
&9.9-7$|$9.9-7
&-3.7-8$|$6.4-8
&36:19$|$51:25
\\[1pt]   
gka5f &501$;$374250&501 &14229$|$48670
&9.9-7$|$9.9-7
&-4.8-8$|$9.8-8
&42:37$|$1:06:37
\\[1pt]

\hline
\mc{7}{l}{QSDP-BIQ (group 3)}\\[1pt]
  
bqp500-2 &501$;$374250&501 &18311$|$66867
&9.9-7$|$9.9-7
&-1.9-7$|$1.2-7
&41:33$|$1:11:30
\\[1pt]   
 
 
bqp500-4 &501$;$374250&501 &14169$|$65580
&9.9-7$|$9.9-7
&-7.8-7$|$1.1-7
&30:04$|$1:10:29
\\[1pt] 
   
 
bqp500-6 &501$;$374250&501 &16428$|$68301
&9.9-7$|$9.9-7
&-2.3-7$|$8.4-8
&36:25$|$1:13:20
\\[1pt]  
  
   
bqp500-8 &501$;$374250&501 &26308$|$107664
&9.9-7$|$9.9-7
&-4.0-7$|$9.5-9
& 1:01:17$|$2:00:06
\\[1pt]  
  
  
bqp500-10 &501$;$374250&501 &16398$|$57221
&9.9-7$|$9.9-7
&-2.8-7$|$8.6-8
&37:22$|$1:06:27
\\[1pt]  
  
gka1f &501$;$374250&501 &14479$|$51294
&9.9-7$|$9.9-7
&-3.6-7$|$7.0-8
&31:05$|$59:17
\\[1pt]  
  
gka2f &501$;$374250&501 &9365$|$60799
&9.9-7$|$9.9-7
&-1.5-6$|$-1.9-9
&18:30$|$1:04:14
\\[1pt]  
  
gka3f &501$;$374250&501 &14175$|$57782
&9.9-7$|$9.9-7
&-3.2-7$|$2.0-8
&30:10$|$1:01:35
\\[1pt]    

gka4f &501$;$374250&501 &13356$|$56588
&9.8-7$|$9.9-7
&-5.8-7$|$-2.0-8
&27:42$|$1:00:10
\\[1pt]  
  
gka5f &501$;$374250&501 &14122$|$58716
&9.9-7$|$9.9-7
&-1.4-7$|$9.3-8
&29:38$|$1:01:13
\\[1pt]

\hline
\end{tabular}
\end{center}
\end{scriptsize}
In the table, ``sGS-isP'' stands for the sGS-isPADMM and ``sP-d'' stands for the sPADMM-4d and sPADMM-5d collectively.
The computation time is in the format of ``hours:minutes:seconds''
\end{table}

\section{Conclusions}
\label{conclusion}

In this paper, by combining  an inexact 2-block majorized sPADMM
and the recent advances in the  inexact symmetric Gauss-Seidel (sGS) technique
 for solving a multi-block convex composite quadratic programming whose objective
contains a nonsmooth term involving only the first block-variable, we have  proposed an inexact multi-block ADMM-type  method (called the sGS-imsPADMM) for solving
general high-dimensional convex composite conic optimization problems to moderate accuracy. One
of the most attractive features of  our proposed method is that it only needs one cycle of
the inexact sGS method, instead of an unknown number of cycles, to solve each of the subproblems involved.
 Our preliminary numerical results for solving $320$ high-dimensional linear and convex quadratic SDP problems with bound constraints, as well as with
a large number of
 linear equality and inequality constraints  have shown that for the vast majority of the tested problems,  the proposed the sGS-imsPADMM is 2 to 3 times faster than the directly extended multi-block PADMM (with no convergence guarantee) even with the aggressive step-length of $1.618$. This is a striking surprise given the fact that although the latter's convergence  is not guaranteed, it is currently the benchmark among first-order methods targeting to solve multi-block linear and quadratic SDPs to modest accuracy.
Our results clearly demonstrate that one does not need to sacrifice speed in exchange for convergence guarantee in developing ADMM-type first order methods, at least for solving high-dimensional linear and convex quadratic SDP problems to moderate accuracy. As mentioned in the Introduction, our goal of designing  the sGS-imsPADMM is to  obtain a
good initial point to warm-start the augmented Lagrangian method
so as to quickly benefit from its fast local  linear convergence.
So the  next important step is to  see how the  sGS-imsPADMM  can be exploited to produce efficient solvers for solving high-dimensional convex composite conic optimization problems to high accuracy. We leave this as our future research topic.

\begin{acknowledgements}
The authors would like to thank Xudong Li and Liuqin Yang  at National University of Singapore for suggestions on the numerical implementations of the algorithms described in this paper, and Caihua Chen at Nanjing University and Ying Cui at National University of Singapore for discussions on the non-ergodic iteration complexity of first-order methods.
\end{acknowledgements}

\end{document}